\documentclass[a4,11pt]{amsart}

\usepackage{lineno,hyperref}

\usepackage{amsmath, amsthm, amsfonts, amssymb}
\usepackage{mathtools}
\usepackage{enumerate}
\usepackage{tensor}
\usepackage{tikz-cd}
\usepackage{pgfplots}
\pgfplotsset{compat=1.18}
\usepackage{comment}

\usepackage{color}
\usepackage[colorinlistoftodos]{todonotes}

\newcommand{\C}{\mathbb{C}}
\newcommand{\N}{\mathbb{N}}

\newcommand{\cB}{\mathcal{B}}
\newcommand{\cU}{\mathcal{U}}
\newcommand{\cM}{\mathcal{M}}
\newcommand{\LL}{\Lambda}
\newcommand{\GG}{\Gamma}

\newcommand{\Tr}{\mathrm{Tr}}
\newcommand{\actson}{{\, \curvearrowright \,}}
\newcommand{\aactson}[1]{{\, \curvearrowright^{#1} \,}}

\newtheorem{thm}{Theorem}[section]
\newtheorem{prop}[thm]{Proposition}

\newtheorem{lem}[thm]{Lemma}
\theoremstyle{definition}
\newtheorem{defn}[thm]{Definition}
\newtheorem{defn/lem}[thm]{Definition/Lemma}
\newtheorem{rem}[thm]{Remark}

\newcommand{\bE}{{\mathbb E}}

\newcommand{\cH}{\mathcal H}
\newcommand{\cK}{\mathcal K}
\newcommand{\cF}{\mathcal F}

\newcommand\Ad{\operatorname{Ad}}

\newcommand\Aut{\operatorname{Aut}}

\newcommand\id{\operatorname{id}}

\newcommand\Span{\operatorname{Span}}

\newcommand{\ovt}{\, \overline{\otimes}\,}
\newcommand{\dpg}{\Delta_p^\Gamma}
\newcommand{\dpl}{\Delta_p^\Lambda}
\newcommand{\talpha}{\tilde{\alpha}}

\title{von Neumann Orbit Equivalence}
\author{Ishan Ishan}
\author{Aoran Wu}

\address{Department of Mathematics, University of Nebraska--Lincoln, 203 Avery Hall, Lincoln, NE 68588, USA}
\email{fishan2@nebraska.edu}

\address{Department of Mathematics, University of Virginia, Kerchof Hall, 141 Cabell Dr, Charlottesville, VA 22903, USA}
\email{aw5dk@virginia.edu}

\AtBeginDocument{%
\def\MR#1{}
}

\begin{document}

\begin{abstract}
We generalize the notion of orbit equivalence to the non-commutative setting by introducing a new equivalence relation on groups, which we call von Neumann orbit equivalence (vNOE). We prove the stability of this equivalence relation under taking free products and graph products of groups. To achieve this, we introduce von Neumann orbit equivalence of tracial von Neumann algebras, show that two countable discrete groups $\Gamma$ and $\Lambda$ are vNOE if and only if the corresponding group von Neumann algebras $L\Gamma$ and $L\Lambda$ are vNOE, and that vNOE of tracial von Neumann algebras is stable under taking free products and graph products of tracial von Neumann algebras.
\end{abstract}

\maketitle

\section{Introduction}

Let $\Gamma$ and $\Lambda$ be two countable discrete groups with free probability measure-preserving actions $\Gamma\actson (X,\mu)$ and $\Lambda\actson(Y,\nu)$ on standard probability measure spaces $(X,\mu)$ and $(Y,\nu)$, respectively. An \textit{orbit equivalence} (OE) for the actions is a measurable isomorphism $\theta:X\to Y$ such that $\theta(\Gamma x)=\Lambda\theta(x)$ for almost every $x\in X$. In this case, the two actions are called \textit{orbit equivalent}. Two groups are said to be \textit{orbit equivalent} if they admit orbit equivalent actions. Singer \cite{Si55} showed that for two free probability measure preserving actions $\Gamma\actson(X,\mu)$ and $\Lambda\actson(Y,\nu)$, being orbit equivalent is equivalent to the existence of an isomorphism $L^\infty(X)\rtimes\Gamma\cong L^\infty(Y)\rtimes\Lambda$ which preserves the Cartan subalgebras $L^\infty(X)$ and $L^\infty(Y)$.  Orbit equivalence theory saw some development in the 1980s (see \cite{OW80,CFW81,Zi84}), and has been an area of active research over the last two decades (see \cite{Fu11,Ga10}). These advances in part have been stimulated by the success of the deformation/rigidity theory approach to the classification of $\mathrm{II}_1$ factors developed by Popa and others (see \cite{Po08,Vae10,Io12}).

The study of orbit equivalence can also be motivated from an entirely different point of view, being a measurable counterpart to quasi-isometry of groups. Gromov \cite{Gro91} introduced measure equivalence (ME) for countable discrete groups as a measurable analogue of quasi-isometry and since then, this notion has proven to be an important tool in geometric group theory with connections to ergodic theory and operator algebras. Two infinite countable discrete groups $\Gamma$ and $\Lambda$ are \textit{measure equivalent} if there is a standard infinite measure space $(\Omega,m)$ with commuting, measure-preserving actions $\Gamma\actson(\Omega,m)$ and $\Lambda\actson(\Omega,m)$, so that both the actions admit finite-measure fundamental domains $Y,X\in\Omega$, that is, $m(Y),m(X)<\infty$ and
$$\Omega=\bigsqcup_{\gamma\in\Gamma}\gamma Y=\bigsqcup_{\lambda\in\Lambda}\lambda X.$$
The space $(\Omega,m)$ is called an ME-\textit{coupling} between $\Gamma$ and $\Lambda$, and the \textit{index} of such a coupling is
$$[\Gamma:\Lambda]_\Omega:=\frac{m(X)}{m(Y)}.$$
Notably, measure equivalence was used by Furman in \cite{Fu99m,Fu99o} to prove strong rigidity results for lattices in higher rank simple Lie groups. ME relates back to OE because of the following fact, observed by Zimmer and Furman: for two discrete groups $\Gamma$ and $\Lambda$, admitting free OE actions is equivalent to having an ME-coupling so that the fundamental domains coincide (see, for example, \cite[\textbf{P$_{\mathrm{ME}}$5}]{Ga05}, \cite[Theorem 2.5]{Fu11}).

If $X \subset \Omega$ is a Borel fundamental domain for the action $\Gamma\actson(\Omega,m)$, then on the level of function spaces, the characteristic function $1_X$ gives a projection in $L^\infty(\Omega, m)$ such that the collection $\{ 1_{\gamma X} \}_{\gamma \in \Gamma}$ forms a partition of unity, i.e., $\sum_{\gamma \in \Gamma} 1_{\gamma X} = 1$. This notion generalizes quite nicely to the non-commutative setting, and using this, Peterson, Ruth, and the first named author, in \cite{IPR}, defined that a \textit{fundamental domain} for an action on a von Neumann algebra $\Gamma \aactson{\sigma} \mathcal M$ is a projection $p \in \mathcal M$ such that $\sum_{\gamma \in \Gamma}\sigma_\gamma(p) = 1$, where the convergence is in the strong operator topology. Using this perspective for a fundamental domain they generalized the notion of measure equivalence by considering actions on non-commutative spaces. 
\begin{defn}[\cite{IPR}]\label{defn:vne}
Two countable discrete groups $\Gamma$ and $\Lambda$ are \textit{von Neumann equivalent} (vNE), written $\Gamma \sim_{\rm vNE} \Lambda$, if there exists a von Neumann algebra $\mathcal M$ with a faithful normal semi-finite trace ${\rm Tr}$ and commuting, trace-preserving actions of $\Gamma$ and $\Lambda$ on $\mathcal M$ such that the $\Gamma$- and $\Lambda$-actions individually admit a finite-trace fundamental domain. The semi-finite von Neumann algebra $\cM$ is called a \textit{von Neumann coupling} between $\Gamma$ and $\Lambda$.
\end{defn}

Like ME, vNE is stable under taking the direct product of groups. But neither ME nor vNE is stable under taking free products. For instance, since any two finite groups are ME (and hence vNE), and amenability is preserved under both ME and vNE, one gets that $\mathbb Z/2\mathbb Z*\mathbb Z/2\mathbb Z$ (amenable) is neither ME nor vNE to $\mathbb Z/3\mathbb Z*\mathbb Z/2\mathbb Z$ (non-amenable). However, as suggested in \cite[Remark 2.28]{MS06}, and proved in \cite[$\mathbf{P}_{\mathrm{ME}}{\bf 6}$]{Ga05}, stability under taking free products hold if one requires the additional assumption that groups are ME with a common fundamental domain. In other words, OE is stable under taking free products. This raises a natural question: {\it Is vNE, with common fundamental domain, stable under taking free products?} We obtain an affirmative answer to this question and introduce the following definition.

\begin{defn}
  Two countable discrete groups $\Gamma$ and $\Lambda$ are said to be \textit{von Neumann orbit equivalent} (vNOE), denoted $\Gamma\sim_{\mathrm{vNOE}}\Lambda$, if there exists a von Neumann coupling between $\Gamma$ and $\Lambda$ with a common fundamental domain.  
\end{defn}

\begin{thm}\label{thm: stability of free products for groups}
    If $\Gamma_i,\Lambda_i, ~i=1,2$ are countable discrete groups such that $\Gamma_i\sim_{\mathrm{vNOE}}\Lambda_i, ~i=1,2$, then $\Gamma_1*\Gamma_2\sim_{\mathrm{vNOE}}\Lambda_1*\Lambda_2$. 
\end{thm}

Green \cite{Green90}, in her Ph.D. thesis, introduced graph products of groups, another important group-theoretical construction. If $\mathcal{G}=(V,E)$ is a simple, non-oriented graph with vertex set $V$ and edge set $E$, then the graph product of a family, $\{\Gamma_v\}_{v\in V}$, of groups indexed by $V$ is obtained from the free product $*_{v\in V}\Gamma_v$ by adding commutator relations determined by the edge set $E$. Depending on the graph, free products and direct products are special cases of the graph product construction. Adapting the ideas of \cite{Ga05}, Horbez and Huang \cite[Proposition 4.2]{HH22} proved the stability of OE under taking graph products over finite simple graphs (see also \cite{DE22}). To further explore the study of graph products within the context of measured group theory, we would like to draw the reader's attention to the article \cite{EH24}. In the current article, we also prove the stability of vNOE under taking graph products.

\begin{thm}\label{thm: stability of graph products for groups}
Let $\mathcal{G}=(V,E)$ be a simple graph with at most countably infinite vertices. Let $\Gamma$ and $\Lambda$ be two graph products over $\mathcal{G}$, with countable vertex groups $\{\Gamma_v\}_{v\in V}$ and $\{\Lambda_v\}_{v\in V}$, respectively. If $\Gamma_v\sim_{\mathrm{vNOE}}\Lambda_v$ for every $v\in V$, then $\Gamma\sim_{\mathrm{vNOE}}\Lambda$.
\end{thm}

In attempting to prove the above theorems, if one tries to adapt the techniques from one of \cite{Ga05, HH22, DE22}, an immediate problem is presented by the lack of ``point perspective" in the theory of von Neumann (orbit) equivalence. The lack of any natural non-commutative analogue of the notion of OE/ME cocycles, or that of measured equivalence relation can be considered as a few problems presented by the lack of point perspective. This often leads one to consider genuinely new techniques and different alternatives (see e.g., \cite{IPR, Ish, Bat, BatMd}). To overcome this obstruction, we introduce the notion of von Neumann orbit equivalence for tracial von Neumann algebras that is ``compatible" with vNOE of groups (see Definition \ref{defn: vNOE vNalgs} and Theorem \ref{thm: gp vNOE iff gp vNalg vNOE}), and prove the analogues of Theorems \ref{thm: stability of free products for groups} and \ref{thm: stability of graph products for groups} at the level of tracial von Neumann algebras.

The notion of von Neumann equivalence admits a generalization in the setting of finite von Neumann algebras \cite[Section 8]{IPR}, and relates to vNE for groups as follows: $\Gamma\sim_{\mathrm{vNE}}\Lambda$ if and only if $L\Gamma\sim_{\mathrm{vNE}}L\Lambda$ \cite[Theorem 1.5]{IPR}. In parallel to this, one might attempt to define two tracial von Neumann algebras to be vNOE if they are vNE and admit a ``common" fundamental domain, and identify a correct meaning of ``common". However, we take a slightly different approach, and motivated by the recently defined notion of measure equivalence of finite von Neumann algebras by Berendschot and Vaes in \cite{BV22}, we introduce the following definition.

\begin{defn}\label{defn: vNOE vNalgs}
    Let $(A,\tau_A)$ and $(B,\tau_B)$ be tracial von Neumann algebras. We say that $(A,\tau_A)$ and $(B,\tau_B)$ are \textit{von Neumann orbit equivalent}, denoted $(A,\tau_A)\sim_{\mathrm{vNOE}}(B,\tau_B)$, if there exists a tracial von Neumann algebra $(Q,\tau_Q)$, a Hilbert $A\ovt Q-B$-bimodule $\cH$, and a vector $\xi\in\cH$ such that 
    \begin{enumerate}
        \item $\langle(a\otimes x)\xi,\xi\rangle=\tau_A(a)\tau_Q(x)$, and $\langle y\xi b,\xi\rangle=\tau_Q(y)\tau_B(b)$ for every $a\in A, ~x,y\in Q,$ and $b\in B$.

        \item $\overline{\Span((A\ovt Q)\xi)}=\cH=\overline{\Span(Q\xi B)}$.
    \end{enumerate}
\end{defn}

We prove in Proposition \ref{vNOE is an equivalence relation} that vNOE is indeed an equivalence relation. We should remark that, in the above definition, $\cH$ can also be considered as an $A-B\ovt Q^{\mathrm{op}}$-bimodule satisfying conditions analogous to the two mentioned in the definition. This essentially is the reason for the symmetry of vNOE, even though the definition seems asymmetric at first. To prove transitivity, inspired by \cite[Lemma 5.11]{BV22}, we establish an equivalent characterization of vNOE in Theorem \ref{thm: rephrase vNOE via homoms}, and show in Theorem \ref{thm: gp vNOE iff gp vNalg vNOE} that $\Gamma\sim_{\mathrm{vNOE}}\Lambda$ if and only if $L\Gamma\sim_{\mathrm{vNOE}}L\Lambda$. Since $L(\Gamma*\Lambda)\cong L\Gamma*L\Lambda$, Theorem \ref{thm: stability of free products for groups} follows from the following theorem, which we prove in Section \ref{sec:vNOE}. We should remark that the above definition, as stated, depends on the choice of the traces $\tau_A$ and $\tau_B$. Outside of the case of finite factors, it is not immediately clear whether the above definition is independent of the choice of the traces for general finite von Neumann algebras.

\begin{thm}\label{thm:free product vNalg}
    If $(A_i,\tau_{A_i}), (B_i,\tau_{B_i}), ~i=1,2$ are tracial von Neumann algebras such that $(A_i,\tau_{A_i})\sim_{\mathrm{vNOE}}(B_i,\tau_{B_i}),~ i=1,2$, then, $(A_1*A_2,\tau_{A_1}*\tau_{A_2})\sim_{\mathrm{vNOE}}(B_1*B_2,\tau_{B_1}*\tau_{B_2})$.
\end{thm}

Similar to free products, one also has that the group von Neumann algebra of a graph product of groups is isomorphic to the (von Neumann algebraic) graph product of the group von Neumann algebras, and hence Theorem \ref{thm: stability of graph products for groups} follows from the following theorem.

\begin{thm}\label{thm:graph product vNalg}
Let $\mathcal{G}=(V,E)$ be a simple graph with at most countably infinite vertices. Let $(A,\tau_A)$ and $(B,\tau_B)$ be two graph products over $\mathcal{G}$, with tracial vertex von Neumann algebras $\{(A_v,\tau_{A_v})\}_{v\in V}$ and $\{(B_v,\tau_{B_v})\}_{v\in V}$, respectively. If $(A_v,\tau_{A_v})\sim_{\mathrm{vNOE}}(B_v,\tau_{B_v})$ for every $v\in V$, then $(A,\tau_A)\sim_{\mathrm{vNOE}}(B,\tau_B)$.
\end{thm}

\begin{rem}
Since graph product over a totally disconnected graph, i.e., a graph with no edges, gives free product, Theorem \ref{thm:free product vNalg} follows from Theorem \ref{thm:graph product vNalg}. In particular, Theorem \ref{thm:graph product vNalg} shows that Theorem \ref{thm:free product vNalg} holds for free products of countably many tracial von Neumann algebras similar to the result obtained by Gaboriau in \cite[$\mathbf{P}_{\mathrm{ME}}{\bf 6^*}$]{Ga05}. Furthermore, Horbez and Huang \cite[Proposition 4.2]{HH22} proved the stability of OE under taking graph products over finite simple graphs, but our result holds for an arbitrary, countably infinite graph product. We include a proof of Theorem \ref{thm:free product vNalg} for two reasons. Firstly, the notation is a little less involved compared to the proof of Theorem \ref{thm:graph product vNalg}. Secondly, for the convenience of a reader who might be interested in the result but is not familiar with graph products.
\end{rem}

In Proposition \ref{prop:vNOE of vNalg implies vNE of vNalg}, we show that vNOE tracial von Neumann algebras are vNE in the sense of \cite{IPR}. We should remark that vNE does not imply vNOE in general. 

In the final section, we obtain a partial analogue of Singer's theorem \cite{Si55} for OE in the setting of vNOE of groups. As noted in \cite[Example 5.2]{IPR},  if $\Gamma$ and $\Lambda$ are countable discrete groups with trace-preserving actions $\Gamma\actson (A,\tau_A)$ and $\Lambda\actson(B,\tau_B)$ on tracial von Neumann algebras $(A,\tau_A)$ and $(B,\tau_B)$, respectively, and if $\theta:B\rtimes\Lambda\to A\rtimes\Gamma$ is a trace-preserving isomorphism such that $\theta(B)=A$, then $\Gamma\sim_{\mathrm{vNOE}}\Lambda$. As a partial converse to this, we prove the following theorem.

\begin{thm}
If $\Gamma$ and $\Lambda$ are countable discrete groups such that $\Gamma\sim_{\mathrm{vNOE}}\Lambda$, then there exist tracial von Neumann algebras $(A,\tau_A),~ (B,\tau_B)$, trace-preserving actions $\Gamma\actson A,~\Lambda\actson B$, and a trace-preserving isomorphism $\theta:B\rtimes\Lambda\to A\rtimes\Gamma$.
\end{thm}

\subsection*{Acknowledgements.}
The authors would like to express deep gratitude to Ben Hayes for his valuable insights, continuous support, and constant encouragement throughout this project. This work has benefited profoundly from discussions with Ben Hayes on numerous occasions and from him encouraging us to look into \cite{BV22}. The project began with the authors' meeting at the ``NCGOA Spring Institute 2023" held at the Vanderbilt University, and we thank the organizers of the conference for the opportunity. A.Wu would also like to acknowledge the support from the NSF CAREER award DMS $\#$214473. Finally, we would like to thank the anonymous referee for helpful suggestions and pointing out an argument to us that helped in simplifying the proof of Theorem \ref{thm:graph product vNalg} and extending it to the case of an arbitrary countably infinite graph product.

\section{Preliminaries and Notations}
We set up the notations and collect some facts that will be needed in this article.  

\subsection{Tracial von Neumann algebras and the standard form}
A tracial von Neumann algebra $A$ is endowed with a trace, i.e.,  a faithful, normal, unital linear functional $\tau:A\to\C$ such that $\tau(xy)=\tau(yx)$ for all $x,y\in A$. The trace $\tau$ induces an inner product on $A$ given by $\langle x,y\rangle=\tau(y^*x), x,y\in A$, and we let $L^2(A)$ denote the Hilbert space completion of $A$ with respect to this inner product. When we view an element $x\in A$ as a vector in $L^2(A)$, we denote it by $\hat x$. An element $x\in A$ defines a bounded linear operator on $L^2(A)$ given by $L_x(\hat{y})=\widehat{xy}, y\in A$, and thus we have a representation of $A$ in $\cB(L^2(A))$, called the \textit{standard representation}. There is also a canonical anti-linear conjugation operator $J:L^2(A)\to L^2(A)$, defined by $J(\hat x)=\widehat{x^*}, x\in A$, and we have that $A'=JAJ$, where $A'$ is the commutant of $M$ inside $\cB(L^2(A))$. By the tracial property of $\tau$, for $x\in A$, the operator $R_x$ on $L^2(A)$ defined by $R_x(\hat{y})=\widehat{yx}, y\in A$ is bounded. If we let $\rho(A)=\{R_x:x\in A\}\subset\cB(L^2(A))$, then $\rho(A)=A'$.


\subsection{Free product and amalgamated free product}
 Let $(A_1,\tau_1)$ and $(A_2,\tau_2)$ be tracial von Neumann algebras. The \textit{free product} of $A_1$ and $A_2$ is the unique, up to isomorphism, tracial von Neumann algebra $(A,\tau)$ containing $A_1$ and $A_2$, and is such that $\tau\vert_{A_i}=\tau_i, i=1,2$, $A$ is generated by $A_1\cup A_2$, and $A_1,A_2$ are \textit{free} inside $A$, i.e., $\tau(a_1a_2\cdots a_k)=0$, whenever $a_i\in A_{n_i}$ with $n_i\neq n_{i+1}$ for all $i\in\{1,\ldots, k-1\}$, $n_i\in\{1,2\}$ and $\tau_{n_i}(a_i)=0$ for all $1\leq i\leq k$. We will say that an element $a_1a_2\cdots a_k$ of the algebraic free product $A_1*_{\mathrm{alg}}A_2$ is an {\it alternating centered word} with respect to $\tau$ if $a_i\in A_{n_i}$ with $n_i\neq n_{i+1}$ for all $i\in\{1,\ldots, k-1\}$, $n_i\in\{1,2\}$ and $\tau_{n_i}(a_i)=0$ for all $1\leq i\leq k$. 

For $i=1,2$, let $(A_i,\tau_i)$ be tracial von Neumann algebras, $Q\subset A_i$ be a common von Neumann subalgebra, and $E_i:A_i\to Q$ be faithful, normal conditional expectations. The \textit{amalgamated free product} $(A, E)=(A_1, E_1)*_Q(A_2,E_2)$ is a pair of a von Neumann algebra $A$ generated by $A_1$ and $A_2$ and a faithful normal conditional expectation $E:A\to Q$ such that $A_1$ and $A_2$ are \textit{freely independent} with respect to $E$: $E(a_1a_2\cdots a_k)=0$ whenever $a_{n_i}\in A_{n_i}$ with $n_i\in\{1,2,\}, E_{n_i}(a_i)=0$ for all $1\leq i\leq k$, and $n_i\neq n_{i+1}$ for all $i\in\{1,\ldots,k-1\}$. An element $a_1a_2\cdots a_k\in A$ will be called an \textit{alternating centered word} with respect to $E$ if $a_{n_i}\in A_{n_i}$ with $n_i\in\{1,2\}, E_{n_i}(a_i)=0$ for all $1\leq i\leq k$, and $n_i\neq n_{i+1}$ for all $i\in\{1,\ldots,k-1\}$.

 For the construction and further details on (amalgamated) free products, we refer the reader to \cite{Voi85,VDN92,P93,U99}.


\subsection{Graph product} Let $\mathcal{G}=(V,E)$ be a simple graph with the vertex set $V$ and the edge set $E\subseteq V\times V\setminus\{(v,v) : v\in V\}$. We assume that the graph $\mathcal{G}$ is non-oriented, i.e., $(v,w)\in E$ if and only if $(w,v)\in E$. A word $v_1v_2\cdots v_n$ of vertices in $V$ is called \textit{reduced} if it satisfies the following property: if there exist $k<l$ such that $v_k=v_l$, then there is some $k<j<l$ such that $(v_k,v_j)\notin E$. Let $\mathcal{G}=(V,E)$ be a simple graph, $(A,\tau)$ be a tracial von Neumann algebra, and $\{(A_v,\tau_v) : v\in V\}$ be a family of tracial von Neumann subalgebras of $(A,\tau)$ such that $\tau\vert_{A_v}=\tau_v$ for all $v\in V$. We say that the family $\{(A_v,\tau_v) : v\in V\}$ is $\mathcal{G}$-\textit{independent} if the following property holds: if $v_1\cdots v_n$ is a reduced word and $a_1,\ldots, a_n\in A$ are such that $a_i\in A_{v_i}$ and $\tau(a_i)=0$, then $\tau(a_1\cdots a_n)=0$. On the other hand, given a simple graph $\mathcal{G}=(V,E)$ and a family of tracial von Neumann algebras $\{(A_v,\tau_v) : v\in V\}$, there is a unique, up to isomorphism, tracial von Neumann algebra $(A,\tau)$, called the \textit{graph product} over $\mathcal{G}$ of the family $\{(A_v,\tau_v) : v\in V\}$, and trace-preserving inclusions $\varphi_v: A_v\hookrightarrow A$ such that the family $\{\varphi_v(A_v):v\in V\}$ is $\mathcal{G}$-independent and generates $A$ as a von Neumann algebra (see \cite{Mlo04, CF17}). We denote the graph product $(A,\tau)$ of the family $\{(A_v,\tau_v) : v\in V\}$ by
\begin{align*}
    (A,\tau)=\bigstar_{v\in V}(A_v,\tau_v).
\end{align*} 

\begin{rem}
    If $\mathcal{G}=(V,E)$ is a simple graph, and $\{\Gamma_v : v\in V\}$ is a family of countable discrete groups, then $L(\bigstar_{v\in V}\Gamma_v)=\bigstar_{v\in V}L\Gamma_v$ (see \cite[Remark 3.23]{CF17}).
\end{rem}



\subsection{Modules over tracial von Neumann algebras}
 For details on the proofs of the facts collected in this subsection, we refer the reader to \cite[Chapter 8]{AP17}.

\begin{defn}
    Given von Neumann algebras $A$ and $B$,
    \begin{enumerate}
        \item a \textit{left A-module} is a pair $(\cH,\pi_A)$, where $\cH$ is a  Hilbert space and $\pi_A:A\to\cB(\cH)$ is a normal unital $*$-homomorphism.

        \item a \textit{right B-module} is a pair $(\cH,\pi_B)$, where $\cH$ is a  Hilbert space and $\pi_B:B\to\cB(\cH)$ is a normal unital $*$-anti-homomorphism, i.e., $\pi_B(xy)=\pi_B(y)\pi_B(x)$ for all $x,y\in B$. In other words, $\cH$ is a left $B^{\rm op}$-module, where $B^{\rm op}$ is the opposite algebra.

        \item an $A-B$-\emph{bimodule} is a triple $(\cH,\pi_A,\pi_B)$ such that $(\cH,\pi_A)$ is a left $A$-module, $(\cH,\pi_B)$ is a right $B$-module, and the representations $\pi_A$ and $\pi_B$ commute. For $\xi\in\cH$, $x\in A$, and $y\in B$, we will write $x\xi y$ instead of $\pi_A(x)\pi_B(y)\xi ~(=\pi_B(y)\pi_A(x)\xi).$
    \end{enumerate}
\end{defn}

\begin{defn}
   Let $(A,\tau_A), (B,\tau_B)$ be tracial von Neumann algebras and let $\cH$ be an $A-B$-bimodule. A vector $\xi\in\cH$ is called
   \begin{enumerate}
   \item \textit{tracial} if $\langle x\xi,\xi\rangle=\tau_A(x)$ for every $x\in A$, and $\langle \xi y,\xi\rangle=\tau_B(y)$ for every $y\in B$.
   \item \textit{bi-tracial} if $\langle x\xi y,\xi\rangle=\tau_A(x)\tau_B(y)$ for all $x\in A, y\in B$.
   \item \textit{cyclic} if $\overline{\mathrm{Span}\{x\xi y:x\in A, y\in B\}}=\cH$.
   \end{enumerate}
\end{defn}

Let $(Q,\tau_Q)$ be a tracial von Neumann algebra. Given two left $Q$-modules $\cH$ and $\cK$, we denote by $_Q\cB(\cH,\cK)$ the space of left $Q$-linear bounded maps from $\cH$ into $\cK$, that is
\begin{align*}
    _Q\cB(\cH,\cK)=\{T\in\cB(\cH,\cK): T(x\xi)=x(T\xi) \text{ for all } x\in Q, \xi\in \cH\}.
\end{align*}
We set $_Q\cB(\cH)={_Q\cB(\cH,\cH)}$. It is straightforward to check that $_Q\cB(\cH)=Q'\cap \cB(\cH)$. Moreover, $_Q\cB(\cH)$ is a semi-finite von Neumann algebra equipped with a specific semi-finite trace $\Tr$, depending on $\tau_Q$. Before stating the result that characterizes $\Tr$, observe that, given $S,T\in{_Q\cB(L^2Q,\cH)}$, we have $TS^*\in{_Q\cB(\cH)}$, and $S^*T\in JQJ$, where $J:L^2Q\to L^2Q$ is the canonical conjugation operator. The following is a translation of \cite[Proposition 8.4.2]{AP17} for left $Q$-modules.

\begin{prop}\label{prop: characterizing trace on commutant}
    If $\cH$ is a left $Q$-module over a tracial von Neumann algebra $(Q,\tau_Q)$, then the commutant $_Q\cB(\cH)=Q'\cap\cB(\cH)$ is a semi-finite von Neumann algebra equipped with a canonical faithful normal semi-finite trace $\Tr$ characterized by the equation
    \begin{align*}
        \Tr(TT^*)=\tau_Q(JT^*TJ)
    \end{align*}
    for every left $Q$-linear bounded operator $T:L^2Q\to\cH$.
\end{prop}

\begin{rem}
    \label{rem: observation about trace on commutant}
    Suppose $(Q,\tau_Q)$ is a tracial von Neumann algebra and $\cH$ is a left $Q$-module. If $\xi\in\cH$ is a tracial vector, then the orthogonal projection $P:\cH\to \overline{\Span(Q\xi)}$ lies in $_Q\cB(\cH)$. Moreover, since $\xi$ is tracial, the operator $U:L^2Q\to\overline{\Span(Q\xi)}$ given by $U\hat{x}=x\xi, x\in Q$ is a unitary. Extending $U$ to an isometry from $L^2Q$ into $\cH$ in an obvious way and applying Proposition \ref{prop: characterizing trace on commutant} to $T=PU:L^2Q\to\cH$ yields that $\Tr(P)=\tau_Q(1)=1$.
\end{rem}


\subsection{Actions on semi-finite von Neumann algebras}\label{sec:actions}
For a semi-finite von Neumann algebra $\cM$ with a faithful normal semi-finite trace $\Tr$, the set $\mathfrak n_{\rm Tr} = \{ x \in \mathcal M \mid {\rm Tr}(x^*x) < \infty \}$ is an ideal. Left multiplication of $\mathcal M$ on $\mathfrak n_{\rm Tr}$ induces a normal faithful representation of $\mathcal M$ in $\mathcal B(L^2(\mathcal M, {\rm Tr}))$, called \textit{the standard representation}, where $L^2(\cM,\Tr)$ is the Hilbert space completion of $\mathfrak n_{\rm Tr}$ under the inner product $\langle a, b \rangle_2 = {\rm Tr}(b^* a)$. 

If $\Gamma \aactson{\sigma} \mathcal M$ is a trace-preserving action of a countable discrete group $\Gamma$ on $\cM$, then $\Gamma$ preserves the $\| \cdot \|_2$-norm on $\mathfrak n_{\rm Tr}$. Therefore, restricted to $\mathfrak n_{\rm Tr}$, the action is isometric with respect to the $\| \cdot \|_2$-norm and hence gives a unitary representation $\sigma^0:\Gamma\to\mathcal U(L^2(\mathcal M, {\rm Tr}))$, called the \textit{Koopman representation}. Considering $\mathcal M \subset \mathcal B(L^2(\mathcal M, {\rm Tr}))$ via the standard representation, we have that the action $\sigma: \Gamma \to {\rm Aut}(\mathcal M, {\rm Tr})$ is unitarily implemented via the Koopman representation, i.e., for $x \in \mathcal M$ and $\gamma \in \Gamma$ we have $\sigma_\gamma(x) = \sigma_\gamma^0 x \sigma_{\gamma^{-1}}^0$ (see \cite[Theorem 3.2]{Haa75}).


\section{Von Neumann Orbit Equivalence}\label{sec:vNOE}

In this section, we define von Neumann orbit equivalence for tracial von Neumann algebras and for countable discrete groups. We shall see that groups are von Neumann orbit equivalent if and only if the corresponding group von Neumann algebras are, and we conclude this section with the proof that von Neumann orbit equivalent tracial von Neumann algebras are von Neumann equivalent in the sense of \cite{IPR}.

\subsection{Von Neumann orbit equivalence for tracial von Neumann algebras}
\begin{thm}\label{thm: rephrase vNOE via homoms}
    Let $(A,\tau_A),\text{ and } (B,\tau_B)$ be tracial von Neumann algebras. Then the following are equivalent.
    \begin{enumerate}
        \item There exists a tracial von Neumann algebra $(Q,\tau_Q)$ and a pointed $A\ovt Q-B$-bimodule $(\cH,\xi)$ such that $\xi\in \cH$ is a cyclic and (bi-)tracial vector for both $A\ovt Q$-module structure, and $Q-B$-bimodule structure. That is, for all $a\in A, b\in B$, and $x\in Q$, \label{item:vNOE of algebras}
        \begin{enumerate}
            \item $\langle (a\otimes x)\xi,\xi\rangle=\tau_A(a)\tau_Q(x)$, and $\langle x\xi b, \xi\rangle =\tau_Q(x)\tau_B(b)$; and
            \item $\overline{\mathrm{Span}((A\ovt Q)\xi)}=\cH=\overline{\mathrm{Span}(Q\xi B)}$.
        \end{enumerate}

        \item There exists a tracial von Neumann algebra $(Q,\tau_Q)$, and a normal $*$-homomorphism $\phi:B\to A\ovt Q$ such that \label{item: condition wrt a homom}
        \begin{enumerate}
            \item $\mathbb E_Q\circ \phi=\tau_B$, where $\mathbb E_Q:A\ovt Q\to Q$ is the normal conditional expectation; and
            \item $\overline{\mathrm{Span}\{x\phi(b):b\in B, x\in Q\}}^{\|\cdot\|_2}=L^2(A\ovt Q)$.
        \end{enumerate}     
    \end{enumerate}
\end{thm}

\begin{proof}
Let $(\cH,Q,\xi)$ be a triple as in (\ref{item:vNOE of algebras}). We thus obtain a canonical unitary $U\colon \cH\to L^{2}(A\overline{\otimes}Q)$ such that $U(y\xi)=\hat y$ for all $y\in A\overline{\otimes}Q$. Hence we can define a right action of $B$ on $L^{2}(A\overline{\otimes}Q)$ by
\[\eta\cdot b=U(U^{*}(\eta)b), \textnormal{ for all $\eta\in L^{2}(A\overline{\otimes}Q)$, $b\in B$}.\]
For $b\in B$, we let $R_{b}\in \cB(L^{2}(A\overline{\otimes}Q))$ be the operator corresponding to right multiplication by $b$.
Since $\cH$ is an $A\overline{\otimes}Q-B$ bimodule,  and $U$ is $A\ovt Q$-linear, so, the right action of $B$ commutes with the left action of $A\overline{\otimes}Q$ on $L^2(A\ovt Q)$, and hence $R_b\in (A\ovt Q)'\cap\cB(L^2(A\ovt Q))$ for every $b\in B$. Since the commutant of $A\overline{\otimes}Q$ acting on $L^{2}(A\overline{\otimes}Q)$ is $\rho(A\overline{\otimes}Q)$ we define $\phi\colon B\to A\overline{\otimes}Q$ as follows: for $b\in B$, $\phi(b)$ is the unique element in $A\ovt Q$ such that $R_{b}=\rho(\phi(b))$. This is directly checked to be a normal $*$-homomorphism. Moreover, by definition of $\phi$, we have that $\eta\cdot b=\eta\phi(b),$ for every $\eta\in L^2(A\ovt Q), b\in B$. Since $\xi$ is $Q$-$B$ bi-tracial, we have for all $b\in B,x\in Q$ that 
\[\tau_{Q}(x)\tau_{B}(b)=\langle x\xi b,\xi\rangle=\langle U^*(\hat x)b,U^*(\hat 1)\rangle=\langle \widehat{x\phi(b)},\hat 1\rangle=\tau(x\phi(b)),\]
where $\tau$ denotes the trace on $A\ovt Q$. Furthermore, since $\tau_Q\circ\mathbb E_Q=\tau$, we have
\begin{align*}
    \tau_Q(\tau_B(b)x)=\tau(x\phi(b))=\tau_Q(\mathbb E_Q(x\phi(b)))=\tau_Q(x\mathbb E_Q(\phi(b))),
\end{align*}
for all $x\in Q, b\in B$, whence it follows that $\mathbb E_{Q}\circ \phi=\tau_{B}$. Finally,
\[\overline{\mathrm{Span}\{x\phi(b):x\in Q,b\in B\}}^{\|\cdot\|_{2}}=U(\overline{\mathrm{Span}\{x\xi b:x\in Q,b\in B\}})=U(\cH)=L^{2}(A\overline{\otimes}Q).\]

Conversely, suppose condition (\ref{item: condition wrt a homom}) holds, and let $\cH=L^{2}(A\overline{\otimes}Q)$, and $\xi=\hat 1$.  Define a right action of $B$ on $\cH$ by $\eta\cdot b=\eta \phi(b)$. Then for $x\in Q$, $b\in B$ we have 
\[\overline{\mathrm{Span}\{x\xi\cdot b:x\in Q,b\in B\}}^{\|\cdot\|_{2}}=\overline{\mathrm{Span}\{x\phi(b):x\in Q,b\in B\}}^{\|\cdot\|_{2}}=L^{2}(A\overline{\otimes}Q),\]
and for $x\in Q,b\in B$ we have
\[\langle x\xi\cdot b,\xi\rangle=\langle \widehat{x\phi(b)},\hat 1\rangle=\tau(x\phi(b))=\tau_Q(x\mathbb E_{Q}(\phi(b)))=\tau_Q(x)\tau_{B}(b).\]
\end{proof}

\begin{rem}\label{rem: span in reverse order}
    Since the operation of taking adjoints is isometric on $L^2(A\ovt Q)$, in condition (\ref{item: condition wrt a homom}) of Theorem \ref{thm: rephrase vNOE via homoms}, one might equivalently require that $\overline{\Span\{\phi(b)x:b\in B,x\in Q\}}^{\|\cdot\|_{2}}=L^{2}(A\overline{\otimes}Q)$.
\end{rem}

\begin{defn}
    Let $(A,\tau_A)$, and $(B,\tau_B)$ be tracial von Neumann algebras. We say that $(A,\tau_A)$ is \textit{von Neumann orbit equivalent} to $(B,\tau_B)$, denoted $(A,\tau_A)\sim_{\mathrm{vNOE}}(B,\tau_B)$, if either of the two equivalent conditions in Theorem \ref{thm: rephrase vNOE via homoms} is satisfied. If $(A,\tau_A)$ is von Neumann orbit equivalent to $(B,\tau_B)$, then the triple $(\cH,Q,\xi)$ or the pair $(Q,\phi)$ of Theorem \ref{thm: rephrase vNOE via homoms} will be called a vNOE-\textit{coupling} between $(A,\tau_A)$ and $(B,\tau_B)$.
\end{defn}

\begin{rem}
As mentioned in the introduction, the above definition depends on the choice of the traces $\tau_A$ and $\tau_B$. When in a situation where the traces are fixed, and there is no risk of confusion, we will often drop them, and simply say that $A$ and $B$ are vNOE and write $A\sim_{\mathrm{vNOE}}B$.
\end{rem}

\begin{prop}\label{vNOE is an equivalence relation}
    Von Neumann orbit equivalence is an equivalence relation.
\end{prop}
\begin{proof}
    If $(A,\tau_A)$ is a tracial von Neumann algebra, then taking $Q=\mathbb C, \cH=L^2(A)$, and $\xi=\hat 1\in L^2(A)$ in condition (\ref{item:vNOE of algebras}) of Theorem \ref{thm: rephrase vNOE via homoms} shows that $A\sim_{\mathrm{vNOE}}A$. To see symmetry, let $(A,\tau_A)$ and $(B,\tau_B)$ be tracial von Neumann algebras satisfying condition (\ref{item:vNOE of algebras}) of Theorem \ref{thm: rephrase vNOE via homoms}, and let $\cH, (Q,\tau_Q), \text{ and } \xi\in \cH$ be as in the condition. Note that we can view $\cH$ as an $A-B\ovt Q^{\mathrm{op}}$-bimodule. Consider the conjugate Hilbert space $\overline{\cH}$, and the corresponding canonical $B\ovt Q^{\mathrm{op}}-A$ bimodule structure on $\overline{\cH}$. Then, it is straightforward to check that the triple $(\overline{\cH},Q^{\mathrm{op}},\bar{\xi})$ satisfies (\ref{item:vNOE of algebras}) of Theorem \ref{thm: rephrase vNOE via homoms} and thus, $B\sim_{\mathrm{vNOE}}A$. To show transitivity, we will use condition (\ref{item: condition wrt a homom}) of Theorem \ref{thm: rephrase vNOE via homoms}. To this end, let $(A,\tau_A), (B,\tau_B)$, and $(C,\tau_C)$ be tracial von Neumann algebras. Let $Q_1, Q_2$, and $\phi_1:B\to A\ovt Q_1, \phi_2:C\to B\ovt Q_2$ be as in (\ref{item: condition wrt a homom}) of Theorem \ref{thm: rephrase vNOE via homoms}. Let $Q=Q_1\ovt Q_2$ and let $\phi:C\to A\ovt Q$ be given by
    \begin{align*}
        \phi(c)=(\phi_1\otimes\mathrm{id}_{Q_2})(\phi_2(c)), ~~~ c\in C,
    \end{align*}
    where $\phi_1\otimes\mathrm{id}_{Q_2}:B\ovt Q_2\to A\ovt Q_1\ovt Q_2$ is the natural extension of $\phi_1:B\to A\ovt Q_1$. Let $\mathbb E_{Q_1}:A\ovt Q_1\to Q_1, \mathbb E_{Q_2}:B\ovt Q_2\to Q_2, \text {and }\mathbb E_Q:A\ovt Q_1\ovt Q_2\to Q_1\ovt Q_2$ be normal conditional expectations. Consider the map $\mathbb E_{Q_1}\otimes\mathrm{id}_{Q_2}:A\ovt Q_1\ovt Q_2\to  Q_1\ovt Q_2$. Note that $\mathbb E_Q=\mathbb E_{Q_1}\otimes\mathrm{\id}_{Q_2}$. Therefore,
    \begin{align*}
        \mathbb E_Q\circ\phi=(\mathbb E_{Q_1}\otimes \id_{Q_2})\circ ((\phi_1\otimes\mathrm{id}_{Q_2})\circ\phi_2)=\mathbb E_{Q_2}\circ \phi_2=\tau_C,
    \end{align*}
    where the second to last equality follows from the fact that the following diagram, since $\bE_{Q_1}\circ\phi_1=\tau_B$, is commutative:
    \[
    \begin{tikzcd}[column sep=small]
B\ovt Q_2 \arrow[d,"\mathbb E_{Q_2}"'] \arrow[rr, "\phi_1\otimes\mathrm{id}_{Q_2}"]& &A\ovt Q_1\ovt Q_2\arrow[d,"\mathbb E_{Q_1}\otimes\mathrm{\id}_{Q_2}"]\\
Q_2\arrow[rr,hook,"1\otimes\mathrm{id}_{Q_2}"']& &Q_1\ovt Q_2
\end{tikzcd}\]

Now, consider $V=\overline{\Span\{\phi(c)x : x\in Q, c\in C\}}^{\|\cdot\|_2},$ and note that $V$ is invariant under multiplication on the right by elements of $Q=Q_1\ovt Q_2$. In the light of Remark \ref{rem: span in reverse order}, it suffices to show that $V=L^2(A\ovt Q)$, and for this, since $V$ is invariant under right multiplication by $Q$, it suffices to show that $A\otimes 1\otimes 1\subseteq V$. Recall that 
$$\overline{\Span\{\phi_2(c)x_2 : c\in C, x_2\in Q_2\}}^{\|\cdot\|_2}=L^2(B\ovt Q_2)\supseteq B\ovt Q_2.$$
Hence, we have
$$\overline{\Span\{(\phi_1\otimes\id_{Q_2})(\phi_2(c)(1\otimes x_2)) : c\in C, x_2\in Q_2\}}^{\|\cdot\|_2}\supseteq (\phi_1\otimes\id_{Q_2})(B\ovt Q_2).$$ 
Moreover, since $\phi_1$ is unital and $\phi_1\otimes\id_{Q_2}$ is a homomorphism, we also have that 
\begin{align*}(\phi_1\otimes\id_{Q_2})(\phi_2(c)(1\otimes x_2))=(\phi_1\otimes\id_{Q_2})(\phi_2(c))(1\otimes x_2) \quad \text{ for all } c\in C, x_2\in Q_2.
\end{align*}
 Finally, since $\overline{\Span\{\phi_1(b)x_1 : b\in B, x_1\in Q_1\}}^{\|\cdot\|_2}=L^2(A\ovt Q_1)\supseteq A\ovt Q_1$, the following computation completes the proof.
     \begin{align*}
         V&\supseteq \overline{\Span\{(\phi_1\otimes\id_{Q_2})(\phi_2(c))(x_1\otimes x_2) : c\in C, x_1\in Q_1, x_2\in Q_2\}}^{\|\cdot\|_2}\\
         &\supseteq \overline{\Span\{(\phi_1\otimes\id_{Q_2})((\phi_2(c)(1\otimes x_2))(x_1\otimes 1) : c\in C, x_1\in Q_1, x_2\in Q_2\}}^{\|\cdot\|_2}\\
         &\supseteq \overline{\Span\{(\phi_1\otimes\id_{Q_2})(b\otimes 1)(x_1\otimes 1): b\in B, x_1\in Q_1\}}^{\|\cdot\|_2}\\
         &\supseteq A\otimes 1\otimes 1.
     \end{align*} 
\end{proof}

Checking $\overline{\Span\{\phi(b)x:b\in B,x\in Q\}}^{\|\cdot\|_2}=L^{2}(A\overline{\otimes}Q)$ might not be easy in general. However, the following lemma simplifies verifying it in certain examples. 

\begin{lem}\label{lem:special subalgebra of A via coupling}
Let $(A,\tau_{A}),(B,\tau_{B}),$ and $(Q,\tau_{Q})$ be  tracial von Neumman algebras. Let $\phi\colon B\to A\overline{\otimes}Q$ be a $*$-homomorphism satisfying $\bE_{Q}\circ \phi=\tau_{B}$, where $\mathbb E_Q:A\ovt Q\to Q$ is the normal conditional expectation. Let $V=\overline{\Span\{\phi(b)x:b\in B,x\in Q\}}^{\|\cdot\|_{2}}\subset L^{2}(A\overline{\otimes}Q)$. Then $N=\{a\in A:a\otimes 1\in V\}$ is an SOT-closed subalgebra of $A$. 
\end{lem}

\begin{proof}
The fact that $N$ is SOT-closed follows from the fact that SOT-convergence in $A$ implies $\|\cdot\|_{2}$-convergence. 
First note that $V$ is invariant under left multiplication by elements of $\phi(B)$ and right multiplication by elements of $Q$. We prove the following claim, whence the lemma follows immediately.

\textbf{Claim:} For $\eta\in V$, and $a\in N$ we have that $\eta(a\otimes 1)\in V$.

\emph{Proof of Claim.} Given $\eta\in V$, and $a\in N$, let $\{x_{n}\}_{n\in\N}\subset \mathrm{Span}\{\phi(b)x:b\in B,x\in Q\}$ be such that $\|x_{n}-\eta\|_{2}\to 0$. Since $a\otimes 1$ is bounded, it follows that 
\begin{align*}\|x_{n}(a\otimes 1)-\eta(a\otimes 1)\|_{2}\leq \|x_{n}-\eta\|_{2} \|a\|\to 0,
\end{align*}
as $n\to\infty$. Since $V$ is $\|\cdot\|_{2}$-closed, it suffices to show that $x_{n}(a\otimes 1)\in V$ for all $n\in \N$. To this end, fix $n\in \N$, and write $x_{n}=\sum_{j=1}^{k}\phi(b_{j})y_{j}$ with $b_{j}\in B,y_{j}\in Q$. Then,
\[x_{n}(a\otimes 1)=\sum_{j=1}^{k}\phi(b_{j})y_{j}(a\otimes 1)=\sum_{j=1}^{k}\phi(b_{j})(a\otimes 1)y_{j},\]
where, in the last equality, we use that $A \text{ and } Q$ commute in $A\overline{\otimes}Q$. Since we already noted that $V$ is invariant under left multiplication by $\phi(B)$ and right multiplication by $Q$, and $a\otimes 1\in V$, it follows that $x_{n}(a\otimes 1)\in V$.     
\end{proof}

\begin{rem}
We do \emph{not} know if $N$ is a $*$-subalgebra. 
\end{rem}

\begin{thm}\label{prop: vNOE preserves free products}
If $(A_i,\tau_{A_i}), (B_i,\tau_{B_i}), ~i=1,2$ are tracial von Neumann algebras such that $(A_i,\tau_{A_i})\sim_{\mathrm{vNOE}}(B_i,\tau_{B_i}),~ i=1,2$, then, $(A_1*A_2,\tau_{A_1}*\tau_{A_2})\sim_{\mathrm{vNOE}}(B_1*B_2,\tau_{B_1}*\tau_{B_2})$.
\end{thm}
\begin{proof}
    Since vNOE is an equivalence relation, it suffices to show that if $A\sim_{\mathrm{vNOE}}B$ and if $(C,\tau_C)$ is another tracial von Neumann algebra, then $A*C\sim_{\mathrm{vNOE}}B*C$. Let $(Q,\tau_Q)$ be a tracial von Neumann algebra and $\phi\colon B\to A\overline{\otimes}Q$ be a $*$-homomorphism as in condition (\ref{item: condition wrt a homom}) of Theorem \ref{thm: rephrase vNOE via homoms}. For tracial von Neumann algebra $(C,\tau_{C})$, we view $(A*C)\overline{\otimes}Q$ as $(A\overline{\otimes}Q)*_{Q}(C\overline{\otimes}Q)$. Define $\widetilde{\phi}_{0}\colon B*_{\textnormal{alg}}C\to(A\overline{\otimes}Q)*_{Q}(C\overline{\otimes}Q)$ by declaring that $\widetilde{\phi}_{0}(b)=\phi(b)*_{Q}1$ for $b\in B$ and $\widetilde{\phi}_{0}(c)=1*_Qc$ for $c\in C$. Let $\bE_Q:(A*C)\ovt Q\to Q$ be the normal conditional expectation. Note that $\bE_{Q}\circ \phi=\tau_{B}$ and $\bE_{Q}|_{C\ovt Q}=\tau_{C}\otimes\id_Q$. Therefore, if $x\in B*_{\textnormal{alg}}C$ is an alternating centered word with respect to $\tau_{B*C}$, then $\widetilde{\phi}_{0}(x)$ is an alternating centered word with respect to $\bE_{Q}$. Hence $\bE_{Q}\circ \widetilde{\phi}_{0}=\tau_{B*C}$. Since $\bE_Q$ is trace preserving, it follows that $\widetilde{\phi}_{0}$ is trace-preserving, and thus extends to a unique trace-preserving $*$-homomorphism $\widetilde{\phi}\colon B*C\to (A\overline{\otimes}Q)*_{Q}(C\overline{\otimes}Q)$. Moreover, by continuity we still have $\bE_{Q}\circ \widetilde{\phi}=\tau_{B*C}$. In light of Remark \ref{rem: span in reverse order}, it thus remains to check that $\overline{\Span\{\widetilde{\phi}(x)y : x\in B*C, y\in Q\}}^{\|\cdot\|_2}=L^2((A*C)\ovt Q)$. To this end, set $V=\overline{\Span\{\widetilde{\phi}(x)y:x\in B*C,y\in Q\}}^{\|\cdot\|_{2}}$. Since $V$ is invariant under right multiplication by $Q$, to show that $V=L^{2}((A*C)\ovt Q)$, it suffices to show that $ (A*C)\otimes 1\subseteq V$. For this, by Lemma \ref{lem:special subalgebra of A via coupling}, it suffices to show that $V$ contains $A\otimes 1$ and $C\otimes 1$. That $C\otimes 1\subseteq V$, follows from the fact that $\widetilde{\phi}$ takes the copy of $C$ in $B*C$ to the copy of $C$ in $(A*C)\overline{\otimes}Q$, and since $\overline{\Span\{\phi(b)y:b\in B,y\in Q\}}^{\|\cdot\|_2}=L^{2}(A\overline{\otimes}Q)$, we also have that $A\otimes 1\subseteq V$. 
\end{proof}

We conclude this subsection by proving Theorem \ref{thm:graph product vNalg}, which we recall below.

\begin{thm}
Let $\mathcal{G}=(V,E)$ be a simple graph, with at most countably infinite vertices. Let $(A,\tau_A)$ and $(B,\tau_B)$ be two graph products over $\mathcal{G}$, with tracial vertex von Neumann algebras $\{(A_v,\tau_{A_v})\}_{v\in V}$ and $\{(B_v,\tau_{B_v})\}_{v\in V}$, respectively. If $(A_v,\tau_{A_v})\sim_{\mathrm{vNOE}}(B_v,\tau_{B_v})$ for every $v\in V$, then $(A,\tau_A)\sim_{\mathrm{vNOE}}(B,\tau_B)$.
\end{thm}
\begin{proof}
 Let $(Q_v,\phi_v)$ be a vNOE-coupling between $(A_v,\tau_{A_v})$ and $(B_v,\tau_{B_v})$, where $\phi_v:A_v\to B_v\ovt Q_v$ is a normal unital $*$-homomorphism satisfying
 \begin{align*}
     \bE_v\circ\phi_v=\tau_{A_v} \qquad \text{ and } \qquad \overline{\Span\{\phi_v(a)x: a\in A_v, x\in Q_v\}}^{\|\cdot\|_2}=L^2(B_v\ovt Q_v),
 \end{align*}
 here, $\bE_v:B_v\ovt Q_v\to Q_v$ is the normal conditional expectation. Define $(Q,\tau_Q)$ to be the tensor product $\ovt_{v\in V}(Q_v,\tau_{Q_v})$. We now define a $*$-homomorphism $\phi:A\to B\ovt Q$ and show that it has the desired properties. To this end, first note that if $v$ and $w$ are two distinct vertices connected by an edge, then $B_v$ commutes with $B_w$ and $Q_v$ commutes with $Q_w$ as $v\neq w$. Thus, $\phi_v(A_v)$ and $\phi_w(A_w)$ commute inside $B\ovt Q$, and therefore, we can define $\phi$ on the algebraic graph product (i.e., the universal unital $*$-algebra generated by the $\{A_v\}_{v\in V}$ subject to the relation that $A_v$ commutes with $A_w$ whenever $v$ and $w$ are connected by an edge). In particular, note that $\phi(a)=\phi_v(a)$, whenever $a\in A_v$. Next, we show that $\phi$ extends to a normal unital $*$-homomorphism on $A$ and satisfies $\bE_Q(\phi(a))=\tau_A(a)$ for all $a\in A$, where $\bE_Q:B\ovt Q\to Q$ is the normal conditional expectation. To see this, we note that for any reduced word $v=v_1\cdots v_n$ of vertices, and elements $a_i\in A_{v_i}$ of trace zero, $1\leq i\leq n$, we have that
 \begin{align*}
     \bE_Q(\phi_{v_1}(a_1)\cdots\phi_{v_n}(a_n))=0.
 \end{align*}
Indeed, for every $i\in\{1,\ldots,n\}$, we have $\bE_{v_i}(\phi_{v_i}(a_i))=\tau_{A_{v_i}}(a_i)=0,$ so, $\phi_{v_i}(a_i)$ can be approximated by linear combinations of elements of the form $b_{v_i}\otimes q_{v_i}$, where $b_{v_i}\in B_{v_i}, q_{v_i}\in Q_{v_i}$ with $\tau_{B_{v_i}}(b_{v_i})=0$. Thus, $\bE_Q(\phi_{v_1}(a_1)\cdots\phi_{v_n}(a_n))=0$ and hence it follows that the map $\phi$, defined on the algebraic graph product, is trace-preserving. Therefore, by \cite[Proposition 3.22]{CF17}, $\phi$ extends to a well-defined normal unital $*$-homomorphism on $A$. Furthermore, if either $a=1$ or $a\in A$ is any reduced operator (see \cite[Definition 3.10]{CF17}), then we have seen that $\bE_Q(\phi(a))=0=\tau_A(a)$. Since the linear span of $1$ and reduced operators is a dense $*$-subalgebra of $A$, we conclude that $\bE_Q(\phi(a))=\tau_A(a)$ for all $a\in A$. Finally, to show that $\overline{\Span\{\phi(a)x:a\in A, x\in Q\}}^{\|\cdot\|_2}=L^2(B\ovt Q)$, it suffices to show that
\begin{align}\label{*}\tag{$*$}
    \overline{\Span(\phi_{v_1}(A_{v_1})\cdots\phi_n(A_{v_n})(1\otimes Q))}^{\|\cdot\|_2}=\overline{\Span(B_{v_1}\cdots B_{v_n}(1\otimes Q))}^{\|\cdot\|_2},
\end{align}
for any reduced word $v=v_1\cdots v_n$ of vertices and for all $n$. We proceed by induction. It is straightforward to see that for any vertex $v$, $\overline{\Span(\phi_v(A_v)(1\otimes Q))}^{\|\cdot\|_2}=L^2(B_v\ovt Q)=\overline{\Span(B_v(1\otimes Q))}^{\|\cdot\|_2}$. Suppose that \eqref{*} holds for any reduced word $v_1\cdots v_n$ of length $n$ and let $v_1\cdots v_nv_{n+1}$ be any reduced word of length $n+1$. Then,
\begin{align*}
  \overline{\Span(\phi_{v_1}(A_{v_1})\cdots\phi_{v_{n+1}}(A_{n+1})(1\otimes Q))}^{\|\cdot\|_2}&=\overline{\Span(\phi_{v_1}(A_{v_1})B_{v_2}\cdots B_{v_{n+1}}(1\otimes Q))}^{\|\cdot\|_2}\\
  &=\overline{\Span(\phi_{v_1}(A_{v_1})(1\otimes Q)(B_{v_2}\cdots B_{v_{n+1}}\otimes 1))}^{\|\cdot\|_2}\\
  &=\overline{\Span(B_{v_1}(1\otimes Q)(B_{v_2}\cdots B_{v_{n+1}}\otimes 1))}^{\|\cdot\|_2}\\
  &=\overline{\Span(B_{v_1}B_{v_2}\cdots B_{v_{n+1}}(1\otimes Q))}^{\|\cdot\|_2}.
\end{align*}
Thus, $(A,\tau_A)\sim_{\mathrm{vNOE}}(B,\tau_B)$.
\end{proof}

\subsection{Von Neumann orbit equivalence for groups} Let $\Gamma\actson^\sigma\cM$ be an action of a countable discrete group $\Gamma$ on a von Neumann algebra $\cM$. A \textit{fundamental domain} for the action is a projection $p\in \cM$ such that the projections $\{\sigma_\gamma(p)\}_{\gamma\in\Gamma}$ are pairwise orthogonal and $\sum_{\gamma\in\Gamma}\sigma_\gamma(p)=1$, where the sum converges in the strong operator topology. Two countable discrete groups $\Gamma$ and $\Lambda$ are said to be \textit{von Neumann equivalent}, denoted $\Gamma\sim_{\mathrm{vNE}}\Lambda$, if there exists a semi-finite von Neumann algebra $(\cM,\Tr)$ with a faithful normal semi-finite trace $\Tr$, and commuting trace-preserving actions $\Gamma\actson^\sigma\cM$ and $\Lambda\actson^\alpha\cM$ such that each action admits a finite-trace fundamental domain. Such an $\cM$ is called a \textit{von Neumann coupling} between $\Gamma$ and $\Lambda$.

\begin{defn}
  Two countable groups $\Gamma$ and $\Lambda$ are said to be \textit{von Neumann orbit equivalent}, denoted $\Gamma\sim_{\mathrm{vNOE}}\Lambda$ if there exists a von Neumann coupling between $\Gamma$ and $\Lambda$ with a common fundamental domain.  
\end{defn}

\begin{thm}\label{thm: gp vNOE iff gp vNalg vNOE}
    For countable discrete groups $\Gamma$ and $\Lambda$, $\Gamma\sim_{\mathrm{vNOE}}\Lambda$ if and only if $L\Gamma\sim_{\mathrm{vNOE}}L\Lambda$.
\end{thm}
\begin{proof}
First suppose that $L\Gamma\sim_{\mathrm{vNOE}}L\Lambda$, and let $(\cH,Q,\xi)$ be a triple as in condition (\ref{item:vNOE of algebras}) of Theorem \ref{thm: rephrase vNOE via homoms}. Set $A=L\Gamma$ and $B=L\Lambda$, and consider $\cM=Q'\cap \cB(\cH)={_Q\cB(\cH)}$. For $\gamma\in \Gamma$, let $u_\gamma\in L\Gamma$ be the corresponding unitary and for $T\in\cB(\cH)$, define $\sigma_\gamma(T)=u_\gamma T u_\gamma^*$. Since $L\Gamma$- and $Q$-actions on $\cH$ commute, it follows that $\cM$ is invariant under $\sigma_\gamma$ for each $\gamma\in\Gamma$, and thus we have an action $\Gamma\actson^\sigma\cM$. Similarly, since $\cH$ is a $Q-B$-bimodule, we have an action $\Lambda\actson^\alpha\cM$ given by $\alpha_s(T)=v_s^*Tv_s, s\in\Lambda, T\in \cM$, where $v_s\in L\Lambda$ is the unitary corresponding to $s\in\Lambda$. It is clear that the actions $\Gamma\actson^\sigma\cM$ and $\Lambda\actson^\alpha\cM$ commute. Let $\Tr$ be the canonical trace on $\cM$ given by Proposition \ref{prop: characterizing trace on commutant}. Since $L\Gamma$- and $Q$-actions on $\cH$ commute, we have that for every $\gamma\in\Gamma$, $u_\gamma$ is a $Q$-linear operator on $\cH$ and hence belongs to $\cM$. Therefore, it now follows from the tracial property that $\Tr(\sigma_\gamma(T))=\Tr(u_\gamma Tu_\gamma^*)=\Tr(T)$ for all $\gamma\in\Gamma, T\in\cM$ and hence $\Gamma\actson^\sigma\cM$ is trace-preserving. Similarly, $\Lambda\actson^\alpha\cM$ too is trace-preserving. 

Let $P\in\cB(\cH)$ be the orthogonal projection from $\cH$ onto $\overline{\Span(Q\xi)}$. It is straightforward to see that $P$ is $Q$-linear and thus, $P\in\cM$. Moreover, it follows from Remark \ref{rem: observation about trace on commutant} that $\Tr(P)=1$. Therefore, it only remains to show that $P$ is a fundamental domain for both $\Gamma$- and $\Lambda$-actions. To see that $P$ is a $\Gamma$-fundamental domain, we first note that, since $\xi$ is tracial for the $A\ovt Q$-module structure, we have, for $a\in A, x\in Q$, that 
    \begin{align*}
        P((a\otimes x)\xi)=\tau_A(a)x\xi.
    \end{align*}

    
    Furthermore, if $a=\sum_{\gamma\in\Gamma}a_\gamma u_\gamma$ is the Fourier series expansion of $a\in A$, then we recall that $\tau_A(a)=a_e$, and thus $P((a\otimes x)\xi)=a_e x\xi$. Therefore,
    \begin{align*}
        \sigma_\gamma(P)((a\otimes x)\xi)&=u_\gamma Pu_\gamma^*((a\otimes x)\xi)\\
        &=u_\gamma P\left(\left(\sum_{g\in\Gamma}a_gu_{\gamma^{-1}g}\otimes x\right)\xi\right)\\
        &=a_\gamma(u_\gamma\otimes x)\xi.
    \end{align*}
    If we let $P_\gamma$ be the orthogonal projection from $\cH$ onto $u_\gamma(\overline{\Span(Q\xi)})$, then it follows from the above calculation that $\sigma_\gamma(P)=P_\gamma$, and it is straightforward to check that the projections $\{P_\gamma\}_{\gamma\in \Gamma}$ are pairwise orthogonal. Moreover, since $\cH=\overline{\Span(A\ovt Q)\xi}$, we also get that $\sum_{\gamma\in\Gamma}\sigma_\gamma(P)=1$ and hence $P$ is a $\Gamma$-fundamental domain. Since we also have that $\xi$ is bi-tracial for the $Q-B$-bimodule structure, we observe that, for $b\in B, x\in Q$,
    \begin{align*}
        P(x\xi b)=\tau_B(b)x\xi.
    \end{align*}
    If $b=\sum_{t\in\Lambda}b_tv_t$ is the Fourier series expansion of $b\in B$, then $\tau_B(b)=b_e$, and hence $P(x\xi b)=b_ex\xi$. For $s\in\Lambda$, let $P_s$ be the orthogonal projection from $\cH$ onto $(\overline{\Span(Q\xi)})v_s$. Now the following calculation shows that $\alpha_s(P)=P_s$, so, $\{\alpha_s(P)\}_{s\in\Lambda}$ are pairwise orthogonal, and thus $P$ is a $\Lambda$-fundamental domain since $\cH=\overline{\Span(Q\xi B)}$.
    \begin{align*}
        \alpha_s(P)(x\xi b)&=v_s^*Pv_s\left(x\xi\sum_{t\in\Lambda}b_tv_t\right)\\
        &=v_s^*(b_{s^{-1}}x\xi)\\
        &=b_{s^{-1}}x\xi v_{s^{-1}}.
    \end{align*}

Conversely, suppose $\Gamma\sim_{\mathrm{vNOE}}\Lambda$, and let $(\cM,\Tr)$ be a von Neumann coupling between $\Gamma$ and $\Lambda$ with common fundamental domain $p\in\cM$ for both $\Gamma\actson^\sigma\cM$ and $\Lambda\actson^\alpha\cM$. Let $ A=L\Gamma, B=L\Lambda, \cH=L^2(\cM,\Tr)\ovt \ell^2(\Lambda), Q=\cM^{\Gamma}\rtimes \Lambda,$ and $\xi=p\otimes \delta_e$. Let $\tau$ be the trace on $\cM^\Gamma$, which we recall is given by $\tau(x)=\Tr(pxp)$ (see \cite[Proposition 4.2]{IPR}). Consider the action of $L\Gamma$ on $\cH$ given by 
$$u_\gamma\eta=(\sigma_\gamma^0\otimes\id)\eta, \qquad \gamma\in\Gamma,\eta\in\cH,$$
where $\sigma_\gamma^0$ is the Koopman representation of $\Gamma$ into $\cU(L^2(\cM,\Tr))$. The action of $\cM^\Gamma$ on $\cH$ is given by 
$$x\eta=(x\otimes\id)\eta,\qquad x\in\cM^\Gamma, \eta\in\cH,$$
and $\Lambda$ acts on $\cH$ on the left by 
$$v_s\eta=(\alpha_s^0\otimes\lambda_\Lambda(s))\eta, \qquad s\in\Lambda,\eta\in\cH,$$
where $\lambda_\Lambda:\Lambda\to\cU(\ell^2\Lambda)$ is the left regular representation, and $\alpha_s^0:\Lambda\to\cU(L^2(\cM,\Tr))$ is the Koopman representation implementing the $\Lambda$-action. Since the $\Gamma$- and $\Lambda$-actions on $\cM$ commute, the actions defined above make $\cH$ into a left $L\Gamma\ovt(\cM^\Gamma\rtimes\Lambda)$-module. Furthermore, for $g\in\Gamma, s\in\Lambda,$ and $x\in\cM^\Gamma$, we have 
    \begin{align*}
    \langle (u_g\otimes xv_s)\xi, \xi\rangle&=\langle (u_g\otimes xv_s)(p\otimes \delta_e), p\otimes \delta_e\rangle\\
    &=\langle \sigma_g(x\alpha_s(p))\otimes\delta_{s},p\otimes \delta_e\rangle\\
    &=\delta_{s,e}\Tr(x\sigma_g(p)p)\\
    &=\delta_{s,e}\delta_{g,e}\Tr(pxp)\\
    &=\delta_{s,e}\delta_{g,e}\tau(x)\\
    &=\tau_A(u_g)\tau_Q(xv_s).
    \end{align*}
    Thus, it follows that $\xi$ is tracial for the left $L\Gamma\ovt(\cM^\Gamma\rtimes\Lambda)$-module structure. For a fixed $s\in\Lambda$, we have
     \begin{align*}
     \overline{\Span\{ (u_g\otimes xv_s)\xi :g\in\Gamma, x\in\cM^\Gamma\}}&=\overline{\Span\{ \sigma_g(x\alpha_s(p))\otimes\delta_{s}:g\in\Gamma, x\in\cM^\Gamma\}}\\
     &=\overline{\Span\{ x\alpha_s(\sigma_g(p))\otimes \delta_s:g\in\Gamma, x\in\cM^\Gamma\}}\\
     &=\overline{\Span\{ \alpha_s(\alpha_{s^{-1}}(x)\sigma_g(p))\otimes \delta_s:g\in\Gamma, x\in\cM^\Gamma\}}\\
     &=\overline{\Span\{ \alpha_s(y\sigma_g(p))\otimes \delta_s: g\in\Gamma, y\in \cM^\Gamma\}}\\
     &=(\alpha_s^0\otimes\lambda_\Lambda(s))(\overline{\Span\{y\sigma_g(p)\otimes\delta_e:g\in\Gamma,y\in\cM^\Gamma\}})\\
     &=L^2(\cM,\Tr)\ovt\mathbb C\delta_s,
     \end{align*}
    where the last equality follows from the fact that $\overline{\Span \{x\sigma_g(p): g\in\Gamma, x \in \cM^{\Gamma}\}}=L^2(\cM,\Tr)$ \cite[Proposition 4.2]{IPR}. Therefore, we have
    \begin{align*}
        \overline{\Span((A\ovt Q)\xi)}&=\overline{\Span\{ (u_g\otimes xv_s)\xi :g\in\Gamma, x\in\cM^\Gamma, s\in\Lambda\}}\\
        &=\overline{\Span(L^2(\cM,\Tr)\ovt\mathbb C\delta_s:s\in\Lambda)}=\cH
    \end{align*}

Finally, the right action of $L\Lambda$ on $\cH$ given by 
\begin{align*}
    \eta v_s=(\id\otimes\rho_\Lambda(s^{-1}))\eta,\qquad s\in\Lambda, \eta\in\cH,
\end{align*}
where $\rho_\Lambda:\Lambda\to\cU(\ell^2\Lambda)$ is the right regular representation, makes $\cH$ into a $Q-L\Lambda$-bimodule. For $x\in\cM^\Gamma$, and $s,t\in\Lambda$, we have
\begin{align*}
    \langle xv_s(p\otimes \delta_e)v_t,p\otimes\delta_e\rangle=\langle x\alpha_s(p)\otimes\delta_{st},p\otimes\delta_e\rangle=\delta_{s,e}\delta_{t,e}\Tr(pxp)=\tau_Q(xv_s)\tau_B(v_t),
\end{align*}
and hence, $\xi$ is a tracial vector for the $Q-B$-bimodule structure. We recall from the proof of \cite[Proposition 4.2]{IPR}, that, since $p$ is $\Lambda$-fundamental domain, we have a direct sum decomposition $L^2(\cM,\Tr)=\sum_{s\in\Lambda}L^2(\cM,\Tr)\alpha_s(p)$. Thus, to show that $\overline{\Span(Q\xi B)}=\cH$, it suffices to show that, for $s\in\Lambda$, $\overline{\Span\{xv_s\xi v_t:x\in\cM^\Gamma,t\in\Lambda\}}=L^2(\cM,\Tr)\alpha_s(p)\ovt\ell^2\Lambda$. To this end, fix $s\in\Lambda$ and note that
\begin{align*}
    \overline{\Span\{xv_s(p\otimes\delta_e)v_t:x\in\cM^\Gamma,t\in\Lambda\}}&=\overline{\Span\{\alpha_s(\alpha_{s^{-1}}(x)p)\otimes\delta_{st}: x\in\cM^\Gamma,t\in\Lambda\}}\\
    &=(\alpha_s^0\otimes\lambda_\Lambda(s))(\overline{\Span\{yp\otimes\delta_t:y\in\cM^\Gamma,t\in\Lambda\}})\\
    &=(\alpha_s^0\otimes\lambda_\Lambda(s))(L^2(\cM,\Tr)p\ovt\ell^2\Lambda)\\
    &=L^2(\cM,\Tr)\alpha_s(p)\ovt\ell^2\Lambda
\end{align*}
\end{proof}

\subsection{Relationship to von Neumann equivalence} \begin{defn}[\cite{IPR}]\label{defn:funddomvn}
Let $A$ and $B$ be tracial von Neumann algebras and let $\cM$ be a semi-finite von Neumann algebra such that $A\subset\cM$ and $B^{\mathrm{op}}\subset \cM$.
\begin{enumerate}
\item  A \textit{fundamental domain for $A$ inside of $\mathcal M$} consists of a realization of  the standard representation $A \subset \mathcal B(L^2(A))$ as an intermediate von Neumann subalgebra $A \subset \mathcal B(L^2(A)) \subset \mathcal M$. The fundamental domain is {\it finite} if finite-rank projections in $\mathcal B(L^2(A))$ are mapped to finite projections in $\mathcal M$. 

\item $\cM$ is a {\it von Neumann coupling between $A$ and $B$} if $B^{\rm op} \subset A' \cap \mathcal M$ and each inclusion $A \subset \mathcal M$ and $B^{\rm op} \subset \mathcal M$ has a finite fundamental domain. 
\end{enumerate}
\end{defn}

\begin{defn}[\cite{IPR}]
Two tracial von Neumann algebras $A$ and $B$ are {\it von Neumann equivalent}, denoted $A \sim_{\mathrm{vNE}} B$, if there exists a von Neumann coupling between them. 
\end{defn}

\begin{prop}\label{prop:vNOE of vNalg implies vNE of vNalg}
    Let $(A,\tau_A)$ and $(B,\tau_B)$ be tracial von Neumann algebras. If $A\sim_{\mathrm{vNOE}} B$, then $A\sim_{\mathrm{vNE}}B$.
\end{prop}
\begin{proof}
Suppose $(\cH,Q,\xi)$ is a triple as in condition (\ref{item:vNOE of algebras}) of Theorem \ref{thm: rephrase vNOE via homoms}. As in the proof of (\ref{item:vNOE of algebras}) implies (\ref{item: condition wrt a homom}) in Theorem \ref{thm: rephrase vNOE via homoms}, let $U:\cH\to L^2(A\ovt Q)$ be the unitary such that $U(x\xi)=\hat x$ for all $x\in A\ovt Q$, and let $\phi:B\to A\ovt Q$ be the $*$-homomorphism obtained therein. Let $\cM=Q'\cap \cB(\cH)=\cB(L^2(A))\ovt Q^{\mathrm{op}}$. We will show that $\cM$ is a von Neumann coupling between $A$ and $B$. It is clear that the inclusion $A\subset \cM$ has a finite fundamental domain. We recall that the argument used in defining $\phi$, shows that we have an inclusion $B^{\mathrm{op}}\subset\cM$ and moreover, since the left $A$- and right $B$-actions on $\cH$ commute, we have that $B^{\mathrm{op}}\subset A'\cap \cM$. Thus, it only remains to show that the inclusion $B^{\mathrm{op}}\subset \cM$ has a finite fundamental domain. To this end, note that, we can also view $\cH$ as an $A-B\ovt Q^{\mathrm{op}}$-bimodule, and $\xi$ is tracial and cyclic for the right $B\ovt Q^{\mathrm{op}}$-module structure. Thus, by the same construction as above, we get an inclusion $B^{\mathrm{op}}\subset {Q^{\mathrm{op}}}'\cap \cB(\cH)=\cB(L^2(B))\ovt Q$. Since $\cM=Q'\cap \cB(\cH)={Q^{\mathrm{op}}}'\cap \cB(\cH)$, we get that the inclusion $B^{\mathrm{op}}\subset\cM$ admits a finite fundamental domain and hence $A\sim_{\mathrm{vNE}}B$.
\end{proof}




  \section{Towards an analogue of Singer's Theorem}

  The main goal of this section is to prove the following theorem.

\begin{thm}
\label{thm: analogue of Singer}
If $\Gamma$ and $\Lambda$ are countable discrete groups such that $\Gamma\sim_{\mathrm{vNOE}}\Lambda$, then there exist tracial von Neumann algebras $(A,\tau_A), (B,\tau_B)$, trace-preserving actions $\Gamma\actson A, \Lambda\actson B$, and a trace-preserving isomorphism $\theta:B\rtimes\Lambda\to A\rtimes\Gamma$.
\end{thm}

Before proving the theorem, we first set up some notations and recall a few facts. Let $\Gamma$ be a countable discrete group and $(M,\tau)$ be a tracial von Neumann algebra. A \emph{1-cocycle} for a trace-preserving action $\Gamma\actson^\alpha(M,\tau)$ is a map $w:\Gamma\to\cU(M)$ that satisfies the following cocycle identity:
\begin{align*}
    w_s\alpha_s(w_t)=w_{st}, \quad s,t\in\Gamma.
\end{align*}
If $\Gamma\actson^\beta(M,\tau)$ is another trace-preserving action, then we say that $\alpha$ and $\beta$ are \emph{cocycle conjugate} if there exists an automorphism $\theta\in\Aut(M,\tau)$ and a 1-cocycle $w:\Gamma\to\cU(M)$ for $\alpha$ such that
\begin{align}
\label{cocycle conjugacy of actions}
    \theta\circ\beta_s\circ\theta^{-1}=\Ad(w_s)\circ\alpha_s,\quad s\in\Gamma.
\end{align}
We recall that if $\Gamma\actson^\alpha(M,\tau)$ and $\Gamma\actson^\beta(M,\tau)$ are cocycle conjugate, then $M\rtimes_\alpha\Gamma\cong M\rtimes_\beta\Gamma$. Indeed, let $\theta\in\Aut(M,\tau)$ and $w:\Gamma\to\cU(M)$ be as in \eqref{cocycle conjugacy of actions}, and consider the map $\Theta:M\rtimes_\alpha\Gamma\to M\rtimes_\beta\Gamma$ given by
\begin{align*}
    \Theta(xu_s)=\Ad(w_s)(\theta(x))v_s, \quad x\in M, s\in\Gamma,
\end{align*}
where, for $s\in\Gamma$, $u_s,v_s$ represent the canonical group unitaries in $M\rtimes_\alpha\Gamma,M\rtimes_\beta\Gamma$, respectively. It is then straightforward to verify that $\Theta$ is an isomorphism.\\

\begin{proof}[Proof of Theorem \ref{thm: analogue of Singer}]
Let $\Gamma\actson^\sigma\cM$ and $\Lambda\actson^\alpha\cM$ be commuting, trace-preserving actions of countable discrete groups $\Gamma$ and $\Lambda$ on a semi-finite von Neumann algebra $\cM$ with a faithful normal semi-finite trace $\Tr$. Let $p\in\cM$ be a finite-trace projection which is a common fundamental domain for both $\Gamma$- and $\Lambda$-actions, that is, $\{\sigma_\gamma(p)\}_{\gamma\in\Gamma}$ are mutually orthogonal and $\sum_{\gamma\in\Gamma}\sigma_\gamma(p)=1$ (SOT), and similarly, $\{\alpha_\lambda(p)\}_{\lambda\in\Lambda}$ are mutually orthogonal and $\sum_{\lambda\in\Lambda}\alpha_\lambda(p)=1$ (SOT). From \cite[Propostion 4.2]{IPR}, there exists a unitary $\cF_p:\ell^2\Gamma\ovt L^2(\cM^\Gamma,\tau)\to L^2(\cM,\Tr)$ such that 
\begin{align*}
\cF_p(\delta_\gamma\otimes x)=\sigma_{\gamma^{-1}}(p)x\qquad \text{for all } \gamma\in\Gamma, x\in\cM^\Gamma,
\end{align*}
where the trace $\tau$ on $\cM^\Gamma$ is given by $\tau(x)=\Tr(pxp),~x\in\cM^\Gamma$\footnote[2]{Note that the formula for $\cF_p$ in \cite[Proposition 4.2]{IPR} requires $x\in\mathfrak{n}_\tau\subset\cM^\Gamma$, where $\tau$ is the semi-finite trace on $\cM^\Gamma$ given by $\tau(x)=\Tr(pxp)$. However, when $\Tr(p)<\infty$, $\tau$ is a finite trace and thus we can have $x\in\cM^\Gamma$ in the formula for $\cF_p$.}. Furthermore, from \cite[Proposition 4.3]{IPR}, there is a trace-preserving isomorphism $\dpg:\cM\rtimes\Gamma\to\cB(\ell^2\Gamma)\ovt\cM^\Gamma$ such that for $\gamma\in\Gamma$ and $x\in \cM$,
\begin{align*}
    \dpg(u_\gamma)=\rho_\gamma\otimes 1,\qquad \dpg(x)=\cF_p^*x\cF_p.
\end{align*}
If we view $\cB(\ell^2\Gamma)\ovt\cM^\Gamma$ as $\cM^\Gamma$-valued $\Gamma\times\Gamma$ matrices, then we have that for all $x\in\cM$, $\dpg(x)=[x_{s,t}]_{s,t}$, where
\begin{align*}
    x_{s,t}=\sum_{\gamma\in\Gamma}\sigma_\gamma(\sigma_{t^{-1}}(p)x\sigma_{s^{-1}}(p))\in\cM^\Gamma.
\end{align*}

Since the actions of $\Gamma$ and $\Lambda$ on $\cM$ commute, we get a well-defined action of $\Lambda$ on $\cM\rtimes\Gamma$, which we denote by $\alpha\rtimes\id_\Gamma$, and it is given by 
\begin{align*}
(\alpha_\lambda\rtimes\id_\Gamma)(xu_\gamma)=\alpha_\lambda(x)u_\gamma, \quad\lambda\in\Lambda,\gamma\in\Gamma,x\in\cM.
\end{align*}
Further, let $\id\otimes\alpha$ be the action of $\Lambda$ on $\cB(\ell^2\Gamma)\ovt\cM^\Gamma$ given by
\begin{align*}
    (\id\otimes\alpha_\lambda)(T\otimes x)=T\otimes\alpha_\lambda(x), \quad \lambda\in\Lambda,T\in\cB(\ell^2\Gamma), x\in\cM^\Gamma.
\end{align*}
Define an action $\talpha$ of $\Lambda$ on $\cB(\ell^2\Gamma)\ovt\cM^\Gamma$ by
\begin{align*}
    \talpha_\lambda=\dpg\circ(\alpha_\lambda\rtimes\id_\Gamma)\circ(\dpg)^{-1},~\lambda\in\Lambda.
\end{align*}
By definition, $\talpha$ is conjugate to $\alpha\rtimes\id_\Gamma$, and hence we get an isomorphism of the crossed products
\begin{align*}
\cM\rtimes(\Gamma\times\Lambda)=(\cM\rtimes\Gamma)\rtimes_{\alpha\rtimes\id_\Gamma}\Lambda\underset{\cong}{\overset{\dpg\rtimes\id_\Gamma}{\longrightarrow}}(\cB(\ell^2\Gamma)\ovt\cM^\Gamma)\rtimes_{\talpha}\Lambda.
\end{align*}
Now  we note that for any $\lambda\in\Lambda$, $\alpha_{\lambda^{-1}}(p)$ is again a $\Gamma$-fundamental domain and hence we can consider the map $\Delta_{\alpha_{\lambda^{-1}}(p)}^\Gamma:\cM\rtimes\Gamma\to\cB(\ell^2)\otimes\cM^\Gamma$ defined analogously to $\dpg$. For any $x\in\cM$ and $s,t\in\Gamma$, by a slight abuse of notation, let $[\dpg(x)]_{s,t}$ be the $(s,t)$-entry in the matrix representation of $\dpg(x)$. Similarly, consider $[\Delta_{\alpha_{\lambda^{-1}}(p)}^\Gamma(x)]_{s,t}$. We observe that
\begin{align*}
  \alpha_\lambda([\Delta_{\alpha_{\lambda^{-1}}(p)}^\Gamma(x)]_{s,t})&=\alpha_\lambda\left(\sum_{\gamma\in\Gamma}\sigma_\gamma(\sigma_{t^{-1}}(\alpha_{\lambda^{-1}}(p))x\sigma_{s^{-1}}(\alpha_{\lambda^{-1}}(p)))\right)\\
  &=\sum_{\gamma\in\Gamma}\sigma_\gamma(\sigma_{t^{-1}}(p)\alpha_\lambda(x)\sigma_{s^{-1}}(p))\\
  &=[\dpg(\alpha_\lambda(x))]_{s,t},
\end{align*}
whence it follows that $\dpg(\alpha_\lambda(x))=(\id\otimes\alpha_\lambda)(\Delta_{\alpha_{\lambda^{-1}}(p)}^\Gamma(x))$ for all $x\in\cM$ and  $\lambda\in\Lambda$. Therefore, for any $x\in \cM, \gamma\in\Gamma$, and $\lambda\in\Lambda$, we have
\begin{align*}
    (\dpg\circ(\alpha_\lambda\rtimes\id_\Gamma))(xu_\gamma)&=\dpg(\alpha_\lambda(x)u_\gamma)\\
    &=\dpg(\alpha_\lambda(x))\dpg(u_\gamma)\\
    &=(\id\otimes\alpha_\lambda)(\Delta_{\alpha_{\lambda^{-1}}(p)}^\Gamma(x))(\rho_\gamma\otimes 1)\\
    &=(\id\otimes\alpha_\lambda)(\cF_{\alpha_{\lambda^{-1}}(p)}^*\cF_p\cF_p^*x\cF_p\cF_p^*\cF_{\alpha_{\lambda^{-1}}(p)})(\rho_\gamma\otimes 1)\\
    &=(\id\otimes\alpha_\lambda)(v_\lambda\dpg(x)v_\lambda^*)(\rho_\gamma\otimes 1)\\
    &=(\id\otimes\alpha_\lambda)(v_\lambda\dpg(xu_\gamma)v_\lambda^*),\\
\end{align*}
where $v_\lambda=\cF^*_{\alpha_{\lambda^{-1}}(p)}\cF_p$. The last equality follows from the fact that $v_\lambda\in\cU(L\Gamma\ovt\cM^\Gamma)$ (see \cite[Proposition 4.4]{IPR}), and hence commutes with $\rho_\gamma\otimes 1$. Therefore, if we let $w_\lambda=(\id\otimes\alpha_\lambda)(v_\lambda)\in\cU(L\Gamma\ovt\cM^\Gamma)$, then we have that
\begin{align*}
    \talpha_\lambda=\dpg\circ(\alpha_\lambda\rtimes\id_\Gamma)\circ(\dpg)^{-1}=\Ad(w_\lambda)\circ(\id\otimes\alpha_\lambda).
\end{align*}

{\bf Claim.} The map $w:\Lambda\to\cU(L\Gamma\ovt\cM^\Gamma)$ defined by $w_\lambda=(\id\otimes\alpha_\lambda)(v_\lambda)$ is a 1-cocycle for $\Lambda\actson^{\id\otimes\alpha}\cB(\ell^2\Gamma)\ovt\cM^\Gamma.$

\begin{proof}[Proof of Claim]
First note that, for any $y\in\mathfrak{n}_{\Tr}\subset\cM$, it is straightforward to verify that
    \begin{align*}
        \cF_p^*(y)=\sum_{\gamma\in\Gamma}\delta_\gamma\otimes y_\gamma,
    \end{align*}
    where
    \begin{align*}
        y_\gamma=\sum_{b\in\Gamma}\sigma_{b\gamma^{-1}}(p)\sigma_b(y).
    \end{align*}
    Since $\Tr(p)<\infty$, so, for any $a\in\Gamma$ and $x\in\cM^\Gamma$, we have that $\Tr(\sigma_{a^{-1}}(p)x)<\infty$ and hence $\sigma_{a^{-1}}(p)x\in\mathfrak{n}_{\Tr}$. Therefore, for any $a\in\Gamma$ and $x\in\cM^\Gamma$, we have
    \begin{align*}
        \cF^*_{\alpha_{\lambda^{-1}}(p)}\cF_p(\delta_a\otimes x)&=\cF^*_{\alpha_{\lambda^{-1}}(p)}(\sigma_{a^{-1}}(p)x)\\
        &=\sum_{\gamma\in\Gamma}\delta_\gamma\otimes\left(\sum_{b\in\Gamma}\sigma_{b\gamma^{-1}}(\alpha_{\lambda^{-1}}(p))\sigma_b(\sigma_{a^{-1}}(p)x)\right)\\
        &=\sum_{\gamma\in\Gamma}\delta_\gamma\otimes\left(\sum_{b\in\Gamma}\sigma_{b\gamma^{-1}}(\alpha_{\lambda^{-1}}(p))\sigma_{ba^{-1}}(p)x\right)
    \end{align*}
    Thus, as an $\cM^\Gamma$-valued $\Gamma\times\Gamma$ matrix, we can write $v_\lambda=[[v_\lambda]_{s,t}]_{s,t},$ where
    \begin{align*}
        [v_\lambda]_{s,t}=\sum_{\gamma\in\Gamma}\sigma_{\gamma s^{-1}}(\alpha_{\lambda^{-1}}(p))\sigma_{\gamma t^{-1}}(p),
    \end{align*}
    and therefore, $w_\lambda$ can be written as an $\cM^\Gamma$-valued $\Gamma\times\Gamma$ matrix $w_\lambda=[[w_\lambda]_{s,t}]_{s,t}$, where
    \begin{align*}
        [w_\lambda]_{s,t}=\alpha_\lambda([v_\lambda]_{s,t})=\sum_{\gamma\in\Gamma}\sigma_{\gamma s^{-1}}(p)\sigma_{\gamma t^{-1}}(\alpha_\lambda(p)).
    \end{align*}
    Finally, the following calculation verifies the cocycle identity for $w$. For $\lambda_1,\lambda_2\in\Lambda$, and $s,t\in\Gamma$, we have
    \begin{align*}
        &~~\quad[w_{\lambda_1}(\id\otimes\alpha_{\lambda_1})(w_{\lambda_2})]_{s,t}\\
        &=\sum_{a\in\Gamma}[w_{\lambda_1}]_{s,a}[(\id\otimes\alpha_{\lambda_{1}})(w_{\lambda_2})]_{a,t}\\
        &=\sum_{a\in\Gamma}\left[\left(\sum_{\gamma\in\Gamma}\sigma_{\gamma s^{-1}}(p)\sigma_{\gamma a^{-1}}(\alpha_{\lambda_1}(p))\right)\left(\sum_{\gamma'\in\Gamma}\sigma_{\gamma' a^{-1}}(\alpha_{\lambda_{1}}(p))\sigma_{\gamma' t^{-1}}(\alpha_{\lambda_1\lambda_2}(p))\right)\right]\\
        &=\sum_{a\in\Gamma}\sum_{\gamma\in\Gamma}\sigma_{\gamma s^{-1}}(p)\sigma_{\gamma a^{-1}}(\alpha_{\lambda_1}(p))\sigma_{\gamma t^{-1}}(\alpha_{\lambda_1\lambda_2}(p))\\
        &=\sum_{\gamma\in\Gamma}\sum_{a\in\Gamma}\sigma_{\gamma s^{-1}}(p)\sigma_{\gamma a^{-1}}(\alpha_{\lambda_1}(p))\sigma_{\gamma t^{-1}}(\alpha_{\lambda_1\lambda_2}(p))\\
        &=\sum_{\gamma\in\Gamma}\sigma_{\gamma s^{-1}}(p)\sigma_{\gamma t^{-1}}(\alpha_{\lambda_1\lambda_2}(p))\\
        &=[w_{\lambda_1\lambda_2}]_{s,t}
    \end{align*}
\end{proof}

It now follows from the above claim that the actions $\talpha$ and $\id\otimes\alpha$ of $\Lambda$ on $\cB(\ell^2\Gamma)\ovt\cM^\Gamma$ are cocycle conjugate, and hence we get the following isomorphisms of the crossed products

\begin{align*}
    (\cB(\ell^2\Gamma)\ovt\cM^\Gamma)\rtimes_{\talpha}\Lambda\underset{\cong}{\overset{\Psi^\Gamma}{\longrightarrow}} (\cB(\ell^2\Gamma)\ovt\cM^\Gamma)\rtimes_{\id\otimes\alpha}\Lambda \underset{\cong}{\overset{\Psi^\alpha}{\longrightarrow}} \cB(\ell^2\Gamma)\ovt(\cM^\Gamma\rtimes_\alpha\Lambda),
    \end{align*}
where the isomorphism $\Psi^\alpha$ is the canonical isomorphism and satisfies
\begin{align*}
    \Psi^\alpha((T\otimes x)u_\lambda)=T\otimes xu_\lambda, \qquad T\in\cB(\ell^2\Gamma), x\in\cM^\Gamma, \lambda\in\Lambda.
\end{align*}

Similarly, starting with the isomorphism $\dpl:\cM\rtimes\Lambda\to\cB(\ell^2\Lambda)\ovt\cM^\Lambda$, and performing the above analysis, yields the following isomorphisms of the crossed products
\begin{align*}
\cM\rtimes(\Gamma\times\Lambda)=(\cM\rtimes\Lambda)\rtimes_{\sigma\rtimes\id_\Lambda}\Gamma\underset{\cong}{\overset{\dpl\rtimes\id_\Lambda}{\longrightarrow}}(\cB(\ell^2\Lambda)\ovt\cM^\Lambda)\rtimes_{\tilde{\sigma}}\Gamma,
\end{align*}
\begin{align*}
(\cB(\ell^2\Lambda)\ovt\cM^\Lambda)\rtimes_{\tilde{\sigma}}\Gamma\underset{\cong}{\overset{\Psi^\Lambda}{\longrightarrow}} (\cB(\ell^2\Lambda)\ovt\cM^\Lambda)\rtimes_{\id\otimes\sigma}\Gamma\underset{\cong}{\overset{\Psi^\sigma}{\longrightarrow}} \cB(\ell^2\Lambda)\ovt(\cM^\Lambda\rtimes_{\sigma}\Gamma).
\end{align*}
Thus, if $\Phi:\cB(\ell^2\Gamma)\ovt(\cM^\Gamma\rtimes_\alpha\Lambda)\to \cB(\ell^2\Lambda)\ovt(\cM^\Lambda\rtimes_\sigma\Gamma)$ is given by
\begin{align*}
    \Phi:=\Psi^\sigma\circ\Psi^\Lambda\circ\theta\circ(\Psi^\Gamma)^{-1}\circ(\Psi^\alpha)^{-1},
\end{align*}
where $\theta:=\dpl\rtimes\id_\Lambda\circ(\dpg\rtimes\id_\Gamma)^{-1}$, then $\Phi$ is an isomorphism and moreover, the following diagram commutes.

\[
    \begin{tikzcd}[column sep=tiny]
\cM\rtimes_\sigma\Gamma \arrow[dd,hook']  \arrow[r, "\dpg"]& \cB(\ell^2\Gamma)\ovt\cM^\Gamma\arrow[d,hook'] &(\cB(\ell^2\Gamma)\ovt\cM^\Gamma)\rtimes_{\id\otimes\alpha}\Lambda\arrow[dr,"\Psi^\alpha","\cong"']\\
&  (\cB(\ell^2\Gamma)\ovt\cM^\Gamma)\rtimes_{\talpha}\Lambda \arrow[dd,"\theta","\cong"']\arrow[ur,"\Psi^\Gamma","\cong"'] &  &\cB(\ell^2\Gamma)\ovt(\cM^\Gamma\rtimes_{\alpha}\Lambda)\arrow[dd,"\Phi","\cong"']\\
\cM\rtimes(\Gamma\times\Lambda)\arrow[ur,"\dpg\rtimes\id_\Gamma","\cong"']\arrow[dr,"\dpl\rtimes\id_\Lambda"',"\cong"] & &\\
&  (\cB(\ell^2\Lambda)\ovt\cM^\Lambda)\rtimes_{\tilde{\sigma}}\Gamma\arrow[dr,"\Psi^\Lambda"',"\cong"] &  &\cB(\ell^2\Lambda)\ovt(\cM^\Lambda\rtimes_{\sigma}\Gamma)\\
\cM\rtimes_\alpha\Lambda \arrow[uu,hook']  \arrow[r, "\dpl"]& \cB(\ell^2\Lambda)\ovt\cM^\Lambda\arrow[u,hook'] & (\cB(\ell^2\Lambda)\ovt\cM^\Lambda)\rtimes_{\id\otimes\sigma}\Gamma\arrow[ur,"\Psi^\sigma"',"\cong"]
\end{tikzcd}\]

Let $\omega_{e_\Gamma,e_\Gamma}\in \cB(\ell^2\Gamma)$ (resp. $\omega_{e_\Lambda,e_\Lambda}\in\cB(\ell^2\Lambda)$) denote the orthogonal projection onto $\C\delta_{e_\Gamma}$ (resp. $\C\delta_{e_\Lambda}$). We make the following observations:
\begin{itemize}
    \item $\Psi^\alpha(\omega_{e_\Gamma,e_\Gamma}\otimes 1)=\omega_{e_\Gamma,e_\Gamma}\otimes 1$ and $\Psi^\sigma(\omega_{e_\Lambda,e_\Lambda}\otimes 1)=\omega_{e_\Lambda,e_\Lambda}\otimes 1$.
    \item $\Psi^\Gamma(x)=x$ for all $x\in\cB(\ell^2\Gamma)\ovt\cM^\Gamma$. This is true because of how $\tilde{\alpha}$ and $\id\otimes\alpha$ are cocycle conjugate. Similarly, $\Psi^\Lambda(y)=y$ for all $y\in\cB(\ell^2\Lambda)\ovt\cM^\Lambda$.
    \item $\dpg(p)=\omega_{e_\Gamma,e_\Gamma}\otimes 1$, $\dpl(p)=\omega_{e_\Lambda,e_\Lambda}\otimes 1$, $(\dpg\rtimes\id_\Gamma)(p)=\omega_{e_\Gamma,e_\Gamma}\otimes 1$, and $(\dpl\rtimes\id_\Lambda)(p)=\omega_{e_\Lambda,e_\Lambda}\otimes 1$. These in particular imply that
    \begin{align*}
    \theta(\omega_{e_\Gamma,e_\Gamma}\otimes 1)=\dpl\rtimes\id_\Lambda((\dpg\rtimes\id_\Gamma)^{-1}(\omega_{e_\Gamma,e_\Gamma}\otimes 1))=\dpl\rtimes\id_{\Lambda}(p)=\omega_{e_\Lambda,e_\Lambda}\otimes 1.
    \end{align*}
\end{itemize}

From the observations above, we obtain that
\begin{align*}
\Phi(\omega_{e_\Gamma,e_\Gamma}\otimes 1)&=\Psi^\sigma(\Psi^\Lambda(\theta((\Psi^\Gamma)^{-1}((\Psi^\alpha)^{-1}(\omega_{e_\Gamma,e_\Gamma}\otimes 1)))))\\
&=\Psi^\sigma(\Psi^\Lambda(\theta(\omega_{e_\Gamma,e_\Gamma}\otimes 1)))\\
&=\Psi^\sigma(\Psi^\Lambda(\omega_{e_\Lambda,e_\Lambda}\otimes 1))\\
&=\omega_{e_\LL,e_\LL}\otimes 1.
\end{align*}

Therefore, we have \begin{align*}
\Phi(\cM^{\GG}\rtimes_\alpha \LL)&=\Phi((\omega_{e_\GG,e_\GG}\otimes 1)(\cB(\ell^2\Gamma)\ovt (\cM^{\GG}\rtimes_\alpha \LL))(\omega_{e_\GG,e_\GG}\otimes 1))\\
&=(\omega_{e_\LL,e_\LL}\otimes 1)(\cB(\ell^2\Lambda)\ovt (\cM^{\LL}\rtimes_\sigma \GG))(\omega_{e_\LL,e_\LL}\otimes 1)\\
&=\cM^{\LL}\rtimes_\sigma \GG
\end{align*}

\end{proof}

\bibliographystyle{amsalpha}
\bibliography{vNOE_Aoran_Ishan.bib}

\providecommand{\bysame}{\leavevmode\hbox to3em{\hrulefill}\thinspace}
\providecommand{\MR}{\relax\ifhmode\unskip\space\fi MR }
\providecommand{\MRhref}[2]{%
  \href{http://www.ams.org/mathscinet-getitem?mr=#1}{#2}
}
\providecommand{\href}[2]{#2}
\begin{thebibliography}{CFW81}

\bibitem[AP21]{AP17}
Claire Anantharaman and Sorin Popa, \emph{An introduction to ${II}_1$ factors},
  \url{https://www.math.ucla.edu/~popa/Books/IIunV15.pdf}, 2021.

\bibitem[Bat23a]{Bat}
Bat-Od Battseren, \emph{Von {N}eumann equivalence and group exactness}, J.
  Funct. Anal. \textbf{284} (2023), no.~4, Paper No. 109786, 12.

\bibitem[Bat23b]{BatMd}
\bysame, \emph{Von {N}eumann equivalence and {$M_d$} type approximation
  properties}, Proc. Amer. Math. Soc. \textbf{151} (2023), no.~10, 4447--4459.
  \MR{4643330}

\bibitem[BV]{BV22}
Tey Berendschot and Stefaan Vaes, \emph{Measure equivalence embeddings of free
  groups and free group factors}, {A}nnales {S}cientifiques de l'{E}cole
  {N}ormale {S}upérieure, To Appear.

\bibitem[CF17]{CF17}
Martijn Caspers and Pierre Fima, \emph{Graph products of operator algebras}, J.
  Noncommut. Geom. \textbf{11} (2017), no.~1, 367--411.

\bibitem[CFW81]{CFW81}
A.~Connes, J.~Feldman, and B.~Weiss, \emph{An amenable equivalence relation is
  generated by a single transformation}, Ergodic Theory Dyn. Syst. \textbf{1}
  (1981), 431--450.

\bibitem[Dem22]{DE22}
{\"O}zkan Demir, \emph{Measurable imbeddings, free products, and graph
  products}, arXiv preprint arXiv:2210.16446 (2022).

\bibitem[EH24]{EH24}
Amandine Escalier and Camille Horbez, \emph{Graph products and measure
  equivalence: classification, rigidity, and quantitative aspects}, arXiv
  preprint arXiv:2401.04635 (2024).

\bibitem[Fur99a]{Fu99m}
Alex Furman, \emph{Gromov's measure equivalence and rigidity of higher rank
  lattices}, Ann. of Math. (2) \textbf{150} (1999), no.~3, 1059--1081.
  \MR{1740986}

\bibitem[Fur99b]{Fu99o}
\bysame, \emph{Orbit equivalence rigidity}, Ann. of Math. (2) \textbf{150}
  (1999), no.~3, 1083--1108. \MR{1740985}

\bibitem[Fur11]{Fu11}
\bysame, \emph{A survey of measured group theory}, Geometry, rigidity, and
  group actions, Chicago Lectures in Math., Univ. Chicago Press, Chicago, IL,
  2011, pp.~296--374. \MR{2807836}

\bibitem[Gab05]{Ga05}
D.~Gaboriau, \emph{Examples of groups that are measure equivalent to the free
  group}, Ergodic Theory Dynam. Systems \textbf{25} (2005), no.~6, 1809--1827.
  \MR{2183295}

\bibitem[Gab10]{Ga10}
Damien Gaboriau, \emph{Orbit equivalence and measured group theory},
  Proceedings of the {I}nternational {C}ongress of {M}athematicians. {V}olume
  {III}, Hindustan Book Agency, New Delhi, 2010, pp.~1501--1527. \MR{2827853}

\bibitem[Gre90]{Green90}
Elisabeth~Ruth Green, \emph{Graph products of groups}, Ph.D. thesis, University
  of Leeds, 1990.

\bibitem[Gro93]{Gro91}
M.~Gromov, \emph{Asymptotic invariants of infinite groups}, Geometric group
  theory, {V}ol.\ 2 ({S}ussex, 1991), London Math. Soc. Lecture Note Ser., vol.
  182, Cambridge Univ. Press, Cambridge, 1993, pp.~1--295. \MR{1253544}

\bibitem[Haa75]{Haa75}
Uffe Haagerup, \emph{The standard form of von {N}eumann algebras}, Math. Scand.
  \textbf{37} (1975), no.~2, 271--283. \MR{407615}

\bibitem[HH22]{HH22}
Camille Horbez and Jingyin Huang, \emph{Measure equivalence classification of
  transvection-free right-angled {A}rtin groups}, J. \'Ec. polytech. Math.
  \textbf{9} (2022), 1021--1067. \MR{4443237}

\bibitem[Ioa13]{Io12}
Adrian Ioana, \emph{Classification and rigidity for von {N}eumann algebras},
  European {C}ongress of {M}athematics, Eur. Math. Soc., Z\"urich, 2013,
  pp.~601--625. \MR{3469148}

\bibitem[IPR24]{IPR}
Ishan Ishan, Jesse Peterson, and Lauren Ruth, \emph{Von {N}eumann equivalence
  and properly proximal groups}, Adv. Math. \textbf{438} (2024), Paper No.
  109481, 43. \MR{4689227}

\bibitem[Ish24]{Ish}
Ishan Ishan, \emph{On von {N}eumann equivalence and group approximation
  properties}, Groups Geom. Dyn. \textbf{18} (2024), no.~2, 737--747.

\bibitem[M{\l}o04]{Mlo04}
Wojciech M{\l}otkowski, \emph{{$\Lambda$}-free probability}, Infin. Dimens.
  Anal. Quantum Probab. Relat. Top. \textbf{7} (2004), no.~1, 27--41.
  \MR{2044286}

\bibitem[MS06]{MS06}
Nicolas Monod and Yehuda Shalom, \emph{Orbit equivalence rigidity and bounded
  cohomology}, Ann. of Math. (2) \textbf{164} (2006), no.~3, 825--878.

\bibitem[OW80]{OW80}
Donald~S. Ornstein and Benjamin Weiss, \emph{Ergodic theory of amenable group
  actions. {I}. {T}he {R}ohlin lemma}, Bull. Amer. Math. Soc. (N.S.) \textbf{2}
  (1980), no.~1, 161--164. \MR{551753}

\bibitem[Pop93]{P93}
Sorin Popa, \emph{Markov traces on universal {J}ones algebras and subfactors of
  finite index}, Invent. Math. \textbf{111} (1993), no.~2, 375--405.
  \MR{1198815}

\bibitem[Pop08]{Po08}
\bysame, \emph{On the superrigidity of malleable actions with spectral gap}, J.
  Amer. Math. Soc. \textbf{21} (2008), no.~4, 981--1000. \MR{2425177}

\bibitem[Sin55]{Si55}
I.~M. Singer, \emph{Automorphisms of finite factors}, Amer. J. Math.
  \textbf{77} (1955), 117--133. \MR{66567}

\bibitem[Ued99]{U99}
Yoshimichi Ueda, \emph{Amalgamated free product over {C}artan subalgebra},
  Pacific J. Math. \textbf{191} (1999), no.~2, 359--392. \MR{1738186}

\bibitem[Vae10]{Vae10}
Stefaan Vaes, \emph{Rigidity for von {N}eumann algebras and their invariants},
  Proceedings of the {I}nternational {C}ongress of {M}athematicians. {V}olume
  {III}, Hindustan Book Agency, New Delhi, 2010, pp.~1624--1650. \MR{2827858}

\bibitem[VDN92]{VDN92}
D.~V. Voiculescu, K.~J. Dykema, and A.~Nica, \emph{Free random variables}, CRM
  Monograph Series, vol.~1, American Mathematical Society, Providence, RI,
  1992, A noncommutative probability approach to free products with
  applications to random matrices, operator algebras and harmonic analysis on
  free groups. \MR{1217253}

\bibitem[Voi85]{Voi85}
Dan Voiculescu, \emph{Symmetries of some reduced free product
  {$C^\ast$}-algebras}, Operator algebras and their connections with topology
  and ergodic theory ({B}u\c steni, 1983), Lecture Notes in Math., vol. 1132,
  Springer, Berlin, 1985, pp.~556--588. \MR{799593}

\bibitem[Zim84]{Zi84}
Robert~J. Zimmer, \emph{Ergodic theory and semisimple groups}, Monographs in
  Mathematics, vol.~81, Birkh\"auser Verlag, Basel, 1984. \MR{776417}

\end{thebibliography}
\end{document}